\documentclass[11pt, letterpaper]{article}
\usepackage[margin=1.25in]{geometry}                
\usepackage{enumitem}
\usepackage{placeins}
\usepackage{graphicx}
\usepackage{stmaryrd}
\usepackage{epstopdf}
\usepackage[utf8]{inputenc}
\usepackage[english]{babel}
\usepackage[square,numbers,sort&compress]{natbib} 
\usepackage{scalerel}[2016/12/29]
\usepackage{bbm}
\usepackage{amsmath,amsfonts,amsthm,amssymb}

\usepackage{lipsum}
	
	\newcommand\blfootnote[1]{%
		\begingroup
		\renewcommand\thefootnote{}\footnote{#1}%
		\addtocounter{footnote}{-1}%
		\endgroup
}

\usepackage{tikz}
\usetikzlibrary{arrows}
\usetikzlibrary{arrows.meta}
\definecolor{wwhhii}{rgb}{1.,1.,1.}
\definecolor{rreedd}{rgb}{1.,0.,0.}
\definecolor{uuuuuu}{rgb}{0.26666666666666666,0.26666666666666666,0.26666666666666666}

\usepackage[bookmarksopen, bookmarksnumbered]{hyperref}
\hypersetup{
	colorlinks=true,
	linkcolor=blue,
	citecolor=blue,
	filecolor=magenta,      
	urlcolor=cyan,
	linktoc=page
}
\usepackage{amsmath}
\usepackage{amsthm}
\usepackage{amssymb}
\usepackage{enumitem}
\usepackage{placeins}
\usepackage{graphicx}
\usepackage{epstopdf}
\usepackage{xcolor}

\theoremstyle{plain}
\newtheorem{theorem}{Theorem}[section]

\newtheorem{prop}[theorem]{Proposition}
\newtheorem{lemma}[theorem]{Lemma}

\theoremstyle{definition}
\newtheorem{definition}[theorem]{Definition}
\newtheorem{remark}[theorem]{Remark}
\newtheorem{fact}[theorem]{Fact}

\newcommand{\Z}{\mathbb{Z}}
\newcommand{\Q}{\mathbb{Q}}
\newcommand{\N}{\mathbb{N}}
\newcommand{\R}{\mathbb{R}}
\newcommand{\p}{\mathbb{P}}
\newcommand{\E}{\mathbb{E}}

\newcommand{\fX}{\mathfrak{X}}

\newcommand{\cl}[1]{{\left\lceil #1 \right\rceil}}
\newcommand{\sset}{\subset}
\newcommand{\lf}{\left}
\newcommand{\rg}{\right}

\newcommand{\mathforall}{\text{ for all }}
\newcommand{\mathand}{\;\text{and}\;}

\newcommand{\ep}{\epsilon}

\newcommand{\sig}{\sigma}

\newcommand{\la}{\lambda}
\newcommand{\al}{\alpha}

\newcommand{\Rd}{\mathbb{R}^4_\uparrow}
\newcommand{\Qd}{\mathbb{Q}^4_\uparrow}

\newcommand{\cA}{\mathcal{A}}

\newcommand{\cC}{\mathcal{C}}

\newcommand{\cF}{\mathcal{F}}

\newcommand{\cL}{\mathcal{L}}

\newcommand{\cS}{\mathcal{S}}

\newcommand{\cW}{\mathcal{W}}

\newcommand{\supp}{\text{supp}}

\usepackage{MnSymbol}

\newcommand{\eqd}{\stackrel{d}{=}}

\newcommand{\X}{\times}

\newcommand{\rev}{\text{rev}}
\newcommand{\id}{\text{id}}

\newcommand{\bx}{\mathbf{x}}
\newcommand{\by}{\mathbf{y}}
\newcommand{\bz}{\mathbf{z}}
\newcommand{\bp}{\mathbf{p}}
\newcommand{\bq}{\mathbf{q}}

\newcommand{\II}[1]{\left \llbracket #1 \right \rrbracket}

\usepackage{parskip}

\title{Last passage isometries for the directed landscape}
		\author{Duncan Dauvergne\footnote{Department of Mathematics, University of Toronto, Toronto, ON, CA} \footnote{Email:  duncan.dauvergne@utoronto.ca}}
\begin{document}

	\maketitle
	
		\begin{abstract}
		Consider the restriction of the directed landscape $\cL(x, s; y, t)$ to a set of the form $\{x_1, \dots, x_k\} \X \{s_0\} \X \R \X \{t_0\}$. We show that on any such set, the directed landscape is given by a last passage problem across $k$ locally Brownian functions. The $k$ functions in this last passage isometry are built from certain marginals of the extended directed landscape. As applications of this construction, we show that the Airy difference profile is locally absolutely continuous with respect to Brownian local time, that the KPZ fixed point started from two narrow wedges has a Brownian-Bessel decomposition around its cusp point, and that the directed landscape is a function of its geodesic shapes.
		\blfootnote{\textit{MSC Class:} 60K35}
		\blfootnote{\textit{Data availability statement:} Data sharing not applicable to this article as no datasets were generated or analysed during the current study.}
	\end{abstract}
	
	\tableofcontents
	\section{Introduction}
	
	Let $f$ be a sequence of continuous functions $f_i:\R\to \R, i \in \Z$. For a nonincreasing cadlag function $\pi$ from $[x, y]$ to the integer interval $\II{m, n}$, henceforth a \textbf{path} from $(x, n)$ to $(y, m)$, define the length of $\pi$ with respect to the environment $f$ by
\begin{equation}
\label{E:pi-length}
	\|\pi\|_f = \sum_{i = m}^n f_i(\pi_i) - f_i(\pi_{i+1}).
	\end{equation}
	Here $\pi_i = \inf \{t \in [x, y] : \pi(t) < i\}$ is the time when $\pi_i$ jumps off of line $i$, and if this set is empty we set $\pi_i = y$. We will think of the space $\R \X \Z$ in matrix coordinates, so that paths move up across the page, see Figure \ref{fig:dis-opt}. For $x \le y \in \R$ and $m \le n \in \Z$, define the \textbf{last passage value across the environment $f$} by
	\begin{equation}
	\label{E:BLPP-multi-intro}
	f[(x, n) \to (y, m)] = \sup_{\pi} \|\pi\|_f,
	\end{equation} 
	where the supremum is over all paths from $(x, n)$ to $(y, m)$.
	A path $\pi$ that achieves this supremum is a \textbf{geodesic}. When the environment is a collection of independent two-sided Brownian motions $B = \{B_i : i \in \Z\}$, the Brownian last passage percolation $(x, m; y, n) \mapsto B[(x, m) \to (y, n)]$ has a four-parameter scaling limit, recently constructed in \cite{DOV}. This limit is the directed landscape $\cL$. It is a random continuous function from the parameter space 
	$$
	\R^4_\uparrow = \{u = (p; q) = (x, s; y, t) \in \R^4 : s < t\}
	$$
	to $\R$. More recently, $\cL$ was shown to be the scaling limit of other integrable models of last passage percolation \cite{dauvergne2021disjoint}, and $\cL$ is conjectured to be the scaling limit of all random growth and random metric models in the Kardar-Parisi-Zhang (KPZ) universality class. For background on last passage percolation and the KPZ universality class, see the books and review articles \cite{romik2015surprising, corwin2012kardar, quastel2011introduction, zygouras2018some} and references therein.
	
	A priori, there is no reason to expect that $\cL$ retains characteristics particular to any one of its prelimits.	However, one might guess that Brownian last passage percolation bears a stronger connection with $\cL$ than other models since $\cL$ is known to have locally Brownian behaviour as we vary $x$ and $y$. In this article we explore this connection. Quite surprisingly, certain marginals of the directed landscape can essentially be expressed as Brownian last passage problems!
	
	For this theorem and throughout the paper we say that $B$ is a $k$-dimensional Brownian motion of variance $\al$ if $B = \sqrt{\al} B'$, where $B'$ is a standard $k$-dimensional Brownian motion. We also write $X \ll Y$ for two random variables $X, Y$ if the law of $X$ is absolutely continuous with respect to the law of $Y$.
	
	\begin{theorem}
		\label{T:main-comparison}
	Let $s < t \in \R, b > 0$ and $x_1 < \dots < x_k \in \N$. Let $\cL$ denote the directed landscape and let $B$ be a collection of $k$ independent Brownian motions of variance $2$. Consider the random continuous functions $f_\cL, f_B:\II{1, k}\X [-b, b] \to \R$ given by
	\begin{align*}
f_\cL(i, y) &= \cL(x_i, s; y, t) - \cL(x_i, s; -b, t) \\
f_B(i, y) &= B[(-b - 1, i) \to (y, 1)] - B[(-b-1, i) \to (-b, 1)]. 
	\end{align*}
	Then $f_\cL \ll f_B$.
	\end{theorem}
The recentering by $\cL(x_i, s; -b, t)$ and $B[(-b-1, i) \to (-b, 1)]$ is necessary to deal with the fact that the last passage values $B[(-b-1, i) \to (y, 1)]$ are increasing in $i$, but the corresponding landscape values are not. Note that the $k=1$ case of Theorem \ref{T:main-comparison} is just local absolute continuity of the Airy$_2$ process with respect to Brownian motion; this was first shown in \cite{CH}. 

Theorem \ref{T:main-comparison} follows from a more refined structural theorem -- Theorem \ref{T:landscape-iso} -- that expresses values of the form $\cL(x_i, s; y, t), i \in \II{1, k}, y \in \R$ as a last passage problem across $k$ locally Brownian functions. To state that theorem and explain how Theorem \ref{T:main-comparison} arises, we need to introduce a few more notions.

	\subsection{The Airy sheet, multi-point last passage, and the RSK isometry}
	
	\FloatBarrier
	
	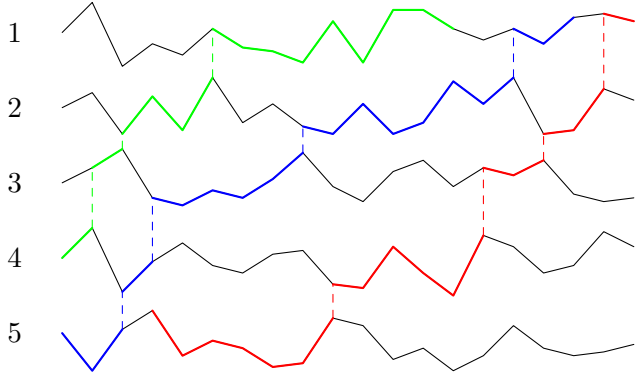
\begin{figure}
		\centering
		\begin{tikzpicture}[line cap=round,line join=round,>=triangle 45,x=4cm,y=5cm]
		\clip(-0.15,-0.15) rectangle (2.15,1.15);
		

		\draw (0.,0.1) node[anchor=east]{$5$};
		\draw (0.,0.3) node[anchor=east]{$4$};
		\draw (0.,0.5) node[anchor=east]{$3$};
		\draw (0.,0.7) node[anchor=east]{$2$};
		\draw (0.,0.9) node[anchor=east]{$1$};

		\draw plot coordinates {(0.1,0.1) (0.2,0.) (0.3,0.11) (0.4,0.16) (0.5,0.04) (0.6,0.08) (0.7,0.06) (0.8,0.01) (0.9,0.02) (1.,0.14) (1.1,0.12) (1.2,0.03) (1.3,0.06) (1.4,0.) (1.5,0.04) (1.6,0.12) (1.7,0.06) (1.8,0.04) (1.9,0.05) (2.,0.07) };
		\draw plot coordinates {(0.1,0.3) (0.2,0.38) (0.3,0.21) (0.4,0.29) (0.5,0.34) (0.6,0.28) (0.7,0.26) (0.8,0.31) (0.9,0.32) (1.,0.23) (1.1,0.22) (1.2,0.33) (1.3,0.26) (1.4,0.2) (1.5,0.36) (1.6,0.33) (1.7,0.26) (1.8,0.28) (1.9,0.37) (2.,0.33) };
		\draw plot coordinates {(0.1,0.5) (0.2,0.54) (0.3,0.59) (0.4,0.46) (0.5,0.44) (0.6,0.48) (0.7,0.46) (0.8,0.51) (0.9,0.58) (1.,0.49) (1.1,0.45) (1.2,0.53) (1.3,0.56) (1.4,0.49) (1.5,0.54) (1.6,0.52) (1.7,0.56) (1.8,0.47) (1.9,0.45) (2.,0.46) };
		\draw plot coordinates {(0.1,0.7) (0.2,0.74) (0.3,0.63) (0.4,0.73) (0.5,0.64) (0.6,0.78) (0.7,0.66) (0.8,0.71) (0.9,0.65) (1.,0.63) (1.1,0.71) (1.2,0.63) (1.3,0.66) (1.4,0.77) (1.5,0.71) (1.6,0.78) (1.7,0.63) (1.8,0.64) (1.9,0.75) (2.,0.72) };
		\draw plot coordinates {(0.1,0.9) (0.2,0.98) (0.3,0.81) (0.4,0.87) (0.5,0.84) (0.6,0.91) (0.7,0.86) (0.8,0.85) (0.9,0.82) (1.,0.93) (1.1,0.82) (1.2,0.96) (1.3,0.96) (1.4,0.91) (1.5,0.88) (1.6,0.91) (1.7,0.87) (1.8,0.94) (1.9,0.95) (2.,0.93) };

		\draw [thick] [red] plot coordinates {(0.4,0.16) (0.5,0.04) (0.6,0.08) (0.7,0.06) (0.8,0.01) (0.9,0.02) (1.,0.14)};
		\draw [dashed] [red] plot coordinates {(1.,0.14) (1.,0.23)};
		\draw [thick] [red] plot coordinates {(1.,0.23) (1.1,0.22) (1.2,0.33) (1.3,0.26) (1.4,0.2) (1.5,0.36)};
		\draw [dashed] [red] plot coordinates {(1.5,0.36) (1.5,0.54)};
		\draw [thick] [red] plot coordinates {(1.5,0.54) (1.6,0.52) (1.7,0.56)};
		\draw [dashed] [red] plot coordinates {(1.7,0.56) (1.7,0.63)};
		\draw [thick] [red] plot coordinates {(1.7,0.63) (1.8,0.64) (1.9,0.75)};
		\draw [dashed] [red] plot coordinates {(1.9,0.75) (1.9,0.95)};
		\draw [thick] [red] plot coordinates {(1.9,0.95) (2.,0.93)};
		
		\draw [thick] [blue] plot coordinates {(0.1,0.1) (0.2,0.) (0.3,0.11)};
		\draw [dashed] [blue] plot coordinates {(0.3,0.11) (0.3,0.21)};
		\draw [thick] [blue] plot coordinates {(0.3,0.21) (0.4,0.29)};
		\draw [dashed] [blue] plot coordinates {(0.4,0.29) (0.4,0.46)};
		\draw [thick] [blue] plot coordinates {(0.4,0.46) (0.5,0.44) (0.6,0.48) (0.7,0.46) (0.8,0.51) (0.9,0.58)};
		\draw [dashed] [blue] plot coordinates {(0.9,0.58) (0.9,0.65)};
		\draw [thick] [blue] plot coordinates {(0.9,0.65) (1.,0.63) (1.1,0.71) (1.2,0.63) (1.3,0.66) (1.4,0.77) (1.5,0.71) (1.6,0.78)};
		\draw [dashed] [blue] plot coordinates {(1.6,0.78) (1.6,0.91)};
		\draw [thick] [blue] plot coordinates {(1.6,0.91) (1.7,0.87) (1.8,0.94)};
		
		\draw [thick] [green] plot coordinates {(0.1,0.3) (0.2,0.38)};
		\draw [dashed] [green] plot coordinates {(0.2,0.38) (0.2,0.54)};
		\draw [thick] [green] plot coordinates {(0.2,0.54) (0.3,0.59)};
		\draw [dashed] [green] plot coordinates {(0.3,0.59) (0.3,0.63)};
		\draw [thick] [green] plot coordinates {(0.3,0.63) (0.4,0.73) (0.5,0.64) (0.6,0.78)};
		\draw [dashed] [green] plot coordinates {(0.6,0.78) (0.6,0.91)};
		\draw [thick] [green] plot coordinates {(0.6,0.91) (0.7,0.86) (0.8,0.85) (0.9,0.82) (1.,0.93) (1.1,0.82) (1.2,0.96) (1.3,0.96) (1.4,0.91)};
		
		\end{tikzpicture}
		\caption{A disjoint optimizer from $((0,0,0.2),5)$ to $((0.7,0.9,1),1)$.}   \label{fig:dis-opt}
	\end{figure}
	The fundamental building block in the directed landscape is the \textbf{Airy sheet} $\cS:\R^2 \to \R$ given by $\cS(x, y) := \cL(x, 0; y, 1)$. The directed landscape is built from independent Airy sheets in an analogous way to how Brownian motion is built from independent normal distributions. The Airy sheet was constructed by understanding how a continuous version of the Robinson-Schensted-Knuth (RSK) correspondence interacts with random inputs. 
	
	Let $\cC^n_0$ be the space of $n$-tuples of continuous functions $f = (f_1, \dots, f_n), f_i:[0,\infty) \to \R$ with $f(0) = 0$. The \textbf{continuous RSK correspondence} is a map $W:\cC^n_0 \to \cC^n_0$. It can be presented in purely geometric terms by looking at \textbf{multi-point last passage percolation}.
Consider an environment $f = (f_i, i \in \Z)$
	and vectors $\bp = (p_1, \dots, p_k), \bq = (q_1, \dots, q_k)$, where $p_i = (x_i, n_i), q_i = (y_i, m_i) \in \R \X \Z$ and $x_i \le x_{i+1}, y_i \le y_{i+1}, n_i \ge n_{i+1}, m_i \ge m_{i+1}$ for all $i$. Let
	\begin{equation}
	\label{E:multi-point-vectors}
	f[\bp \to \bq ]= \sup_\pi \sum_{i=1}^k \|\pi\|_f
	\end{equation}
	where the supremum is over all \textbf{disjoint $k$-tuples} of paths $\pi = (\pi_1, \dots, \pi_k)$, where each $\pi_i$ goes from $(x_i,n)$ to $(y_i, m)$, and $\pi_i(z) < \pi_{i+1}(z)$ for $z \in (x_i, y_i) \cap (x_{i+1}, y_{i+1})$. We call a $k$-tuple that achieves this supremum a \textbf{(disjoint) optimizer}.
	For certain choices of $\bx, \by$, this supremum is taken over the empty set, in which case we set $f[\bp \to \bq ]= -\infty$.
	 For $f \in \cC^n_0$, define $Wf \in \cC^n_0$ by
	\begin{equation}
	\label{E:Wf-definition}
	\sum_{i=1}^k Wf_i(y) = f[(0, n)^k \to (y, 1)^k]
	\end{equation}
	for $k \in \II{1, n}, t \in [0, \infty)$. Here and throughout we write $p^k$ for the vector consisting of $k$ copies of a point $p$. Remarkably, the RSK map $W$ is an \textbf{isometry} between the upper and lower boundaries of $f$. For $\bx \in \R^k_\le := \{\bx \in \R^k : x_1 \le \dots \le x_k\}$, we let $(\bx, m)$ denote the $k$-tuple of points $(x_1, m), \dots, (x_k, m)$.
For any vectors $(\bx, n), (\by, 1)$ with $0 \le x_1$, we have
	\begin{equation}
	\label{E:Wf-isometry}
	f[(\bx, n) \to (\by, 1)] = Wf[(\bx, n) \to (\by, 1)].
	\end{equation}
The isometry \eqref{E:Wf-isometry} was shown in \cite[Proposition 4.1]{DOV}, though closely related formulas had been previously observed by Biane, Bougerol, and O'Connell \cite{biane2005littelmann} and Noumi and Yamada \cite{noumi2002tropical}. 

The RSK correspondence behaves well with certain random inputs. In particular, if $B\in \cC^n_0$ is a sequence of independent Brownian motions, then  $WB$ is a sequence of nonintersecting Brownian motions, see \cite{o2002representation, o2003conditioned}. In the KPZ scaling limit, $WB$ converges to the \textbf{parabolic Airy line ensemble} $\cW$, constructed by Pr\"ahofer and Spohn \cite{prahofer2002scale}, and realized as a sequence of infinity many nonintersecting, locally Brownian functions by Corwin and Hammond \cite{CH}.
	
	The Airy sheet was built in \cite{DOV} by understanding how the isometry \eqref{E:Wf-isometry} passes to the KPZ limit for single points, see Definition \ref{D:Airy-sheet}. In \cite{dauvergne2021disjoint}, this analysis was extended to multiple points to construct an extended Airy sheet and an extended directed landscape. To describe the construction, we need to talk about metric notions for $\cL$. 
	
	The value $\cL(p; q) = \cL(x, s; y, t)$ is best thought of as a distance between two points $p$ and $q$ in the space-time plane. 
Here $x, y$ are spatial coordinates and $s, t$ are time coordinates. We cannot move backwards or instantaneously in time, so $\cL(x, s; y, t)$ is not defined for $s \ge t$.  Unlike with an ordinary metric, $\cL$ is not symmetric and may take negative values. As in last passage percolation, it also satisfies the triangle inequality backwards:
	\begin{equation}
	\label{E:triangle}
	\cL(p; r) \ge \cL(p; q) + \cL(q; r) \qquad \mbox{ for all }(p; r), (p; q), (q; r) \in \Rd.
	\end{equation}
	Just as in true metric spaces, we can define path lengths in $\cL$, see \cite[Section 12]{DOV}. In $\cL$, a \textbf{path} from $(x, s)$ to $(y, t)$ is a continuous function $\pi:[s, t] \to \R$ with $\pi(s) = x$ and $\pi(t) = y$, with \textbf{length}
	\begin{equation}
	\label{E:length-L}
	\|\pi\|_\cL=\inf_{k\in \N}\inf_{s=t_0<t_1<\ldots<t_k=t}\sum_{i=1}^k\cL(\pi(t_{i-1}),t_{i-1};\pi(t_i),t_i)\,.
	\end{equation}
	This is analogous to defining the length of a curve in Euclidean space by piecewise linear approximation.
	A path $\pi$ is a \textbf{geodesic} if $\|\pi\|_\cL$ is maximal among all paths with the same start and end points. Equivalently, a geodesic is any path $\pi$ with $\|\pi\|_\cL = \cL(\pi(s),s;\pi(t),t)$. Almost surely, geodesics exist between every pair of points $(x, s), (y, t)$ with $s < t$.
	
	Now, for $\bx, \by \in \R^k_\le, s < t$ let
	\begin{equation}
	\label{E:multi-point-L}
	\cL(\bx, s; \by, t) = \sup_\pi \sum_{i=1}^k \|\pi_i\|_\cL,
	\end{equation}
	where the supremum is over all $k$-tuples of paths $\pi_i:[s,t] \to \R$ with $\pi_i(s) = x_i, \pi_i(t) = y_i$ and $\pi_i(r) < \pi_{i+1}(r)$ for all $r \in (s, t)$. This extension of $\cL$ is called the \textbf{extended directed landscape}, abbreviated as extended landscape. The extended landscape is continuous on its domain, and satisfies a limiting RSK isometry. Define an infinite collection of functions $W \cL_i:\R \to \R, i \in \N$ by 
	\begin{equation}
\label{E:WL-definition}
\sum_{i=1}^k W\cL_i(y) = \cL(0^k, 0; y^k, 1).
	\end{equation}
	Then $W \cL$ is a parabolic Airy line ensemble and for any $\bx, \by \in \R^k_\le$, we have
	\begin{equation}
	\label{E:cLW}
	\cL(\bx, 0; \by, 1) = W\cL[\bx \to \by].
	\end{equation}
Of course, one needs to make sense of the right side of \eqref{E:cLW}, see Theorem \ref{T:extended-sheet-char} for details.
	
	\subsection{Other isometries}
	
	In addition to $Wf$, for a given $f \in \cC^n_0$ there are other interesting environments $g \in \cC^n_0$ that satisfy the isometric property \eqref{E:Wf-isometry}. Our main theorems in this paper come from exploring these other isometries in the limit $\cL$.
	
	One of the standard methods for constructing $Wf$ (via iterated Pitman transforms) yields a set of environments $W_\tau f \in \cC^n$ satisfying \eqref{E:Wf-isometry} indexed by permutations $\tau \in S_n$.
	The environments $W_\tau f$ can also be described in terms of certain multi-point last passage values in analogy with \eqref{E:Wf-definition}. Moreover, if $B$ is a collection of independent Brownian motions, then while $W_\tau B$ does not have the tractable nonintersecting structure of $WB$,
	its paths are still locally absolutely continuous with respect to Brownian motion, as is the case with $WB$. See Section \ref{S:properties-BLPP} for more details.
	
	This story has an analogue in the directed landscape. The environments we construct will no longer satisfy \eqref{E:cLW} for all possible $\bx, \by$, but rather only for $\bx$ with entries in a particular finite set $\{x_1, \dots, x_k\}$. On the other hand, these environments are in one sense significantly simpler than $W \cL$: they will consist of only $k$ lines. As in the finite setting, while these isometric environments no longer have the integrable or nonintersecting structure of $W \cL$, they are still locally Brownian and their lines can be described via certain marginals of the extended directed landscape, similarly to \eqref{E:WL-definition}.
	
	To set up our main theorem, we need a notion of last passage from  $-\infty$. Let $f = (f_1, \dots, f_k), f_i:\R \to \R$ be an environment of continuous functions. For $I = \{i_1 <  \dots < i_\ell \}\sset \II{1, k}$ and a $k$-tuple $\bp$, define
	\begin{equation}
	\label{E:W-LPP}
	f[(-\infty, I) \to \bp] = \lim_{z \to -\infty} f[(z, I) \to \bp] + \sum_{i \in I} f_i(z).
	\end{equation}
	In \eqref{E:W-LPP} and throughout we write $(z, I)$ for the $|I|$-tuple $(z, i_1), \dots, (z, i_\ell)$. This type of last passage can equivalently be defined using a supremum over disjoint paths, see Remark \ref{R:path-def-infty}. Now let $\bx \in \R^k_\le$ with $x_1 < \dots < x_k$. For a set $I \sset \II{1, k}$, we write $\bx^I$ for the vector in $\R^{|I|}$ consisting only of coordinates $x_i$ of $\bx$ with $i \in I$.
	Define a sequence of $k$ functions $W_\bx \cL = \{W_\bx \cL_i:\R\to \R, i \in \II{1, k}\}$ by the formula
	\begin{equation}
	\label{E:WLx}
	\sum_{i=1}^\ell W_\bx \cL_i(y) = \cL(\bx^{\II{1, \ell}}, 0; y^\ell, 1), \qquad \text{ for } \ell \in \II{1, k}, y \in \R.
	\end{equation}
	\begin{theorem}
		\label{T:landscape-iso}
	Almost surely, for any $\bx \in \R^k_\le$ with $x_1 < \dots < x_k$ we have the following:
	\begin{enumerate}
		\item (Asymptotics and stability) For any $i \in \II{1, k}$, we have
		\begin{equation}
		\label{E:asymptotics}
		\lim_{z \to \pm\infty} \frac{W_\bx \cL_i(z) + z^2}{z} = 2x_i. 
		\end{equation} 
		In particular, this implies that for $y_0 \in \R$, there exists a random $Z_0(y_0) \in (-\infty, y_0) \cap \Z$ such that
		\begin{equation}
		\label{E:WxLx-stable}
		W_\bx \cL[(-\infty, I) \to (\by, 1)] = W_\bx \cL[(z, I) \to (\by, 1)] + \sum_{i \in I} \cW_i(z) 
		\end{equation}
		for all $I \sset \II{1, k}, z \le Z_0$ and $\by \in \R^{|I|}_\le$ with $y_1 \ge y_0$.
		\item (Locally Brownian) Conditional on $W_\bx \cL(0)$ the function $W_\bx \cL(\cdot) - W_\bx \cL(0)$ is locally absolutely continuous with respect to a $k$-dimensional Brownian motion $B$ of variance $2$. In other words, if we let $B, \cL$ be independent then on any interval $[a, b]$, we have $(W_\bx \cL(0), W_\bx \cL|_{[a, b]} - W_\bx \cL(0)) \ll (W_\bx \cL(0),B|_{[a, b]})$.
		\item (Isometry) For any $I \sset \II{1, k}$ and $\by \in \R^{|I|}_\le$, we have
		$$
	W_\bx \cL[(-\infty, I) \to (\by, 1)] = \cL(\bx^I, 0; \by, 1).
		$$
	\end{enumerate}
	\end{theorem}

\begin{remark}
	\begin{enumerate}
		\item Theorem \ref{T:landscape-iso} allows us to represent complicated marginals of the directed landscape and the Airy sheet in terms of last passage in a finite environment of locally Brownian continuous functions. In particular, this allows us to show that almost sure properties of Brownian last passage percolation across \emph{finitely many} lines hold in $\cL$. We discuss a few consequences of this in Section \ref{S:consequences}.
		\item One could set up a version of Theorem \ref{T:landscape-iso} for $\bx \in \R^k_\le$ without the strict ordering condition. This is a technical extension that we do not pursue here. We point out that the case where $x_i = 0$ for all $i$ is implicit in Proposition 5.9 in \cite{dauvergne2021disjoint}. In this special case, $W_\bx \cL$ returns the top $k$ lines of $W \cL$.
		\item Even though we have fixed the times in Theorem \ref{T:landscape-iso} equal to $0$ and $1$, the result also holds for arbitrary times $s < t$ by invariance properties of $\cL$. 
		\item We believe the construction in Theorem \ref{T:landscape-iso} should be useful for studying Busemann functions for the extended landscape. Consider the limiting field
		\begin{equation}
		\label{E:Busemann-field}
		B_\cL(\bx, \by) = \lim_{t \to \infty} \cL(t\bx, -t; 0, \by) - \cL(t\bx, -t; 0,0^k),
		\end{equation}
		which should give the Busemann function in a direction $\bx \in \R^k_\le$ between locations $(0, \by)$ and $(0, 0^k)$. Taking a limit of the construction in Theorem \ref{T:landscape-iso} should yield insight into the nature $B_\cL$. Indeed, for any $\bx \in \R^k_\le, \by \in \R^k_\le$ and $I \sset \II{1, k}$, we should have
		\begin{equation}
		\label{E:WLx-as}
		B_\cL(\bx^I, \by) = W^\infty_\bx \cL [\bx^I \to \by], \qquad \text{ where } \qquad \sum_{i=1}^\ell W^\infty_\bx \cL_i(y) := B_\cL(\bx^{\II{1, \ell}}, y^\ell).
		\end{equation}
		Moreover, each of the lines $W^\infty_\bx \cL_i$ should be an independent, variance $2$ Brownian motion with drift $2 x_i$. We do not attempt to make this limiting picture rigorous, or to justify the existence of the limit in \eqref{E:Busemann-field}. 
		
		Busemann functions have been studied in detail in last passage models, yielding remarkable insights about the geometry of infinite geodesics and other last passage structures, e.g. see \cite{cator2012busemann, georgiou2017stationary, janjigian2019geometry, seppalainen2021busemann}.
		At the level of single points, a Busemann function structure closely related to \eqref{E:WLx-as} was shown for exponential last passage percolation by Fan and Sepp\"al\"ainen \cite{fan2018joint}, building on work of Ferrari and Martin \cite{ferrari2007stationary} studying stationary measures in multi-type tasep. In \cite{fan2018joint}, the analogue of the first equality in \eqref{E:WLx-as} is distributional. The lines corresponding to $W^\infty_\bx \cL_i$ are constructed via certain `multiclass processes' which, a priori, are unrelated to multi-point last passage values.
	\end{enumerate}
\end{remark}
	
\subsection{Consequences}
\label{S:consequences}

In addition to Theorem \ref{T:main-comparison}, we present three other fairly straightforward consequences of Theorem \ref{T:landscape-iso}.  
First, fix $x_1 < x_2$ and define the \textbf{Airy difference profile}
$$
A^{x_1, x_2}(y) = \cL(x_2, 0; y, 1) -  \cL(x_1, 0; y, 1).
$$
The function $A^{x_1, x_2}$ is a continuous increasing function, and hence is the CDF of a random measure $\mu_{x_1, x_2}$. Moreover, basic symmetries of $\cL$ imply that the process $A^{x_1, x_2} - \E A^{x_1, x_2}$ is stationary and 
$$
\E A^{x_1, x_2}(y) = - x_2^2 - x_1^2 + 2y(x_2 - x_1),
$$
so on average, $A^{x_1, x_2}$ increases linearly. However, this is not the case for individual realizations of $A^{x_1, x_2}$; in fact, $\mu_{x_1, x_2}$ is supported on a lower dimensional set.

In \cite{basu2019fractal}, Basu, Ganguly and Hammond showed that the support of $\mu_{x_1, x_2}$ a.s.\ has Hausdorff dimension $1/2$, and in \cite{bates2019hausdorff}, Bates, Ganguly and Hammond proved that this support coincides with a set of exceptional events for geodesics in $\cL$. Moving beyond a Hausdorff dimension estimate, Ganguly and Hegde \cite{ganguly2021local} showed that $A^{x_1, x_2}$ can be represented as a `patchwork quilt' of objects that are locally absolutely continuous with respect to the running maximum of a Brownian motion. The techniques used in \cite{ganguly2021local} have similarities with our methods, i.e. they also use properties of iterated Pitman transforms to study $A^{x_1, x_2}$.

As a consequence of Theorem \ref{T:landscape-iso} in the special case when $\bx = (x_1, x_2)$, we recover all of these previous results and more. In particular, we show that $A^{x_1, x_2}$ is locally absolutely continuous with respect to the running maximum of a Brownian motion without any need for a patchwork quilt.

\begin{theorem}
	\label{T:Difference}
	Fix $\bx = (x_1, x_2)$ with $x_1 < x_2$.
	\begin{enumerate}
		\item For all $y \in \R$,
		$$
		A^{x_1, x_2}(y)  = \sup_{z \le y} W_\bx \cL_2 (z) - W_\bx \cL_1(z)= \sup_{z \le y} \cL((x_1, x_2), 0; z^2, 1) - 2\cL(x_1, 0; z, 1).
		$$
		\item For any compact set $[a, b]$, the law of $A^{x_1, x_2}|_{[a, b]}$ is absolutely continuous with respect to the law of the running maximum $M$ of a Brownian motion $B$ on $[a-1, b]$ of variance $4$, i.e. $M(x) = \sup_{a-1 \le y \le x} B(y)$.
		\item The support of $\mu_{x_1, x_2}$ is the set of $y \in \R$ where 
		$$
		\cL((x_1, x_2), 0; y^2, 1) = \cL(x_1, 0; y, 1) + \cL(x_2, 0; y, 1).
		$$
		Equivalently, this is the of set of $y \in \R$ where there exist geodesics $\pi_1$ from $(x_1, 0) \to (y, 1)$ and $\pi_2$ from $(x_2, 0) \to (y, 1)$ such that $\pi_1(r) < \pi_2(r)$ for all $r \in [0, 1)$.
	\end{enumerate}	
\end{theorem}
The disjoint geodesic characterization of $\supp(\mu_{x_1, x_2})$ in Theorem \ref{T:Difference}.$3$ was shown in \cite{basu2019fractal}. We have included it above to highlight how it can alternately be obtained as an immediate consequence of our main theorem and one of the main results of \cite{dauvergne2021disjoint}. 

As a consequence of Theorem \ref{T:Difference}, we can show that $\cL$ can be reconstructed from only the \emph{shapes} of its geodesics, without any information about their lengths. 
\begin{theorem}
	\label{T:geodesic-frame}
	Let $\Qd = \Rd \cap \Q^4$. For every $u = (p, q) \in \Qd$, let $\pi_\cL(u)$ be the a.s. unique $\cL$-geodesic from $p$ to $q$, and for $s < t \in \Q$ let $\cF_{s, t}$ be the $\sig$-algebra generated by $\{\pi_\cL(x, s; y, t) : (x, y) \in \Q^2\}$ and all null sets. Then  $\cL(\cdot, s; \cdot, t):\R^2 \to \R$ is $\cF_{s, t}$-measurable.
	
	In particular, we can a.s. reconstruct the entire directed landscape $\cL$ using only the information from $\pi_\cL(u), u \in \Qd$.
\end{theorem}

We thank B\'alint Vir\'ag for pointing out the key step in the proof of Theorem \ref{T:geodesic-frame} and the reference \cite{taylor1966exact}.\footnote{Since we first started working with the directed landscape, B\'alint Vir\'ag and I have long been interested in these sort of reconstruction questions for $\cL$. There are many interesting ones. One we particularly like (which would significantly strengthen Theorem \ref{T:geodesic-frame}) is whether a landscape value $\cL(p; q)$ can be reconstructed from only the shape of the geodesic from $p$ to $q$.}

Our final consequence concerns the KPZ fixed point, constructed by Matetski, Quastel, and Remenik \cite{matetski2016kpz}. The KPZ fixed point is a Markov process $\mathfrak{h}_t,t \in [0, \infty)$ taking values in the space of upper semicontinuous functions $h:\R \to \R \cup \{-\infty\}$ with a growth bound at $\pm \infty$. It is the scaling limit of random growth models in the KPZ universality class. By results of \cite{nica2020one, dauvergne2021scaling}, it is related to the directed landscape $\cL$ by the following formula. Letting $h_0:\R \to \R \cup \{-\infty\}$ denote the initial condition of the KPZ fixed point, we can write
\begin{equation}
\label{E:ht-def}
\mathfrak{h}_t(y) = \max_{x \in \R} h_0(x) + \cL(x, 0; y, t).
\end{equation}
In \cite{hammond2019patchwork}, Hammond (see also \cite{hammond2016brownian, calvert2019brownian}) demonstrated how for fixed $t$, $\mathfrak{h}_t$ can be decomposed as a patchwork quilt made of Brownian pieces. More precisely, he showed that under mild assumptions on the initial condition, $\mathfrak{h}_t$ has the following description:
\begin{itemize}
	\item There exists an ordered (finite or countable) random sequence $\dots A_{-1} < A_0 < A_1 < \dots$  which is finite when restricted to any compact set, and an (unordered) random sequence $p_i, i \in \Z$, such that
	\begin{equation}
	\label{E:ht-cusps}
	\mathfrak h_t(x) = \sum_{i \in \Z} \mathbf{1}(x \in [A_i, A_{i+1}))[Y_i(x) + p_i].
	\end{equation}
	\item Each of the functions $Y_i, i \in \Z$ is locally absolutely continuous with respect to a Brownian motion of variance $2$.
\end{itemize}
In this description, each patch $Y_i(x) + p_i = \cL(y_i, 0; x, t) + h_0(y_i)$ for some strictly increasing sequence $y_i$. 

More recently, Sarkar and Vir\'ag \cite{sarkar2021brownian} showed that in fact the KPZ fixed point started from any initial condition is locally absolutely continuous with respect to Brownian motion.
From this point of view, the cusp points $A_i$ are not seen. However, we still expect interesting behaviour at these points. Using Theorem \ref{T:landscape-iso}, we can identify what is actually happening at these cusp points. For simplicity, we restrict ourselves to the easiest nontrivial case when the initial condition $h_0$ is combination of two narrow wedge initial conditions.
In this case there is exactly one cusp, and the KPZ fixed point has the following basic structure.

\begin{fact}
	\label{F:KPZFP-structure}
	Let $\mathfrak{h}_t$ denote the KPZ fixed point at time $t$ started from an initial condition $h_0$ which is equal to $-\infty$ except at two points $p_1 < p_2$ where $h_0(p_1) = a_1, h_0(p_2) = a_2$ for some  $a_1, a_2 \in \R$. Then there exists $A \in \R$ such that
	$$
	\mathfrak h_t(y) = [\cL(p_1, 0; y, 1) + a_1]\mathbf{1}(y < A) + [\cL(p_2, 0; x, 1) + a_2]\mathbf{1}(y \ge A).
	$$
\end{fact}

We are concerned with understanding the joint law of $(A, \mathfrak{h}_t)$.

\begin{theorem}
	\label{T:KPZ-fixed-point}
	Let $X \in \R$ be any random variable with a positive Lebesgue density everywhere, let $B:\R\to \R$ be a two-sided standard Brownian motion, let $R:\R\to \R$ be a two-sided Bessel-$3$ process. That is, $R = \|X\|_2$, where $X$ is a $3$-dimensional standard Brownian motion. Suppose that all $3$ objects are independent. Then with $\mathfrak{h}_t$ as in the setting of Fact \ref{F:KPZFP-structure}, for any $t > 0$ and any compact interval $I \sset \R$, we have $(A, \mathfrak{h}_t(\cdot)|_I - \mathfrak{h}_t(A)) \ll (X, [B + R](-X + \cdot)|_I)$.
\end{theorem}

Theorem \ref{T:KPZ-fixed-point} shows that conditional on its cusp location $A$, $\mathfrak{h}_t(A + \cdot) - \mathfrak h_t(A)$ is locally absolutely continuous with respect to a two-sided Brownian-Bessel process.  If we start from a more general initial condition, then there will be many cusp points as in \eqref{E:ht-cusps}. If we condition on all of these cusp points, then similar ideas could be used to show that the process $\mathfrak{h}_t$ is locally absolutely continuous with respect to a Brownian motion started at the cusp closest to $0$, plus a string of independent Bessel-$3$ bridges running between all cusp points. For brevity, we do not pursue this technical extension here.
	
	\section{Properties of last passage percolation}
	\label{S:properties-BLPP}
	
	\subsection{Basics}
	Recall the definition of multi-point last passage from \eqref{E:multi-point-vectors}. 
 Multi-point last passage percolation satisfies the following useful \textbf{metric composition law}. Its proof is immediate from the definitions.
	
	\begin{prop}
		\label{P:MC-law}
		Let $f = (f_i, i \in \Z)$ be a sequence of continuous functions and let $j \in \Z$. For any endpoints $\bp \in (\R \X \{\dots, j + 2, j + 1\})^k, \bq \in (\R \X \{j, j - 1, \dots \})^k$ we have that
		\begin{equation*}
		f[\bp \to \bq] = \max_{\bz \in \R^k_\le} f[\bp \to (\bz, j+1)] + f[(\bz, j) \to \bq].
		\end{equation*}
		Similarly, for any $x \in \R$ and endpoints $\bp \in ((-\infty, x] \X \Z)^k, \bq \in ([x, \infty) \X \Z)^k$ we have that 
		\begin{equation*}
		f[\bp \to \bq] = \max_{I \in \Z^k_\ge} f[\bp \to (x, I)] + f[(x, I) \to \bq].
		\end{equation*}
		Here and throughout $\Z^k_\ge = \{I \in \Z^k : I_1 \ge I_2 \dots \ge I_k\}$. 
	\end{prop}

We will also use the following facts about optimizers. First, suppose that $\pi, \pi'$ are paths from $(x, n)$ to $(y, m)$ and $(x', m')$ to $(y', m')$, respectively. We write $\pi \le \pi'$ and say that $\pi'$ is \textbf{to the right} of $\pi$ if:
\begin{itemize}[nosep]
	\item $x \le x', y \le y', n \le n', m \le m'$.
	\item For $z \in [x, y] \cap [x', y']$, we have $\pi(z) \le \pi'(z)$.
\end{itemize}
For $k$-tuples of paths $\pi, \pi'$, we write $\pi \le \pi'$ if $\pi_i \le \pi'_i$ for all $i \in \II{1, k}$.
Then we have the following, from \cite{dauvergne2021disjoint}. The first part is Lemma 2.2 from that paper, and part (ii) is a slight variant of Lemma 2.3, whose proof goes through verbatim.

\begin{prop}
	\label{P:basic-optimizers}
	Let $f = (f_i, i \in \Z)$, and let $(\bp, \bq)$ be a pair of $k$-tuples such that there is at least one disjoint $k$-tuple (of paths) from $\bp$ to $\bq$.
	\begin{enumerate}[nosep, label=(\roman*)]
		\item  There always exists an optimizer $\pi$ in $f$ from $\bp$ to $\bq$ such that for any optimizer $\tau$ from $\bp$ to $\bq$, we have $\tau \le \pi$. We call $\pi$ the \textbf{rightmost optimizer} from $\bp$ to $\bq$. 
		\item Let $(\bp', \bq')$ be a $k'$-tuple of endpoints such that there is at least one disjoint $k'$-tuple from $\bp'$ to $\bq'$, and let $s \in \Z$.
		
		 Let $(x_i, n_i), (x_i', n_i'), (y_i, m_i), (y_i', m_i')$ denote the coordinates of $\bp, \bp', \bq, \bq'$, and define $x_i, y_i, m_i, n_i = \infty$ for integers $i > k$ and $x_i, y_i, m_i, n_i = - \infty$ for integers $i \le 0$. Similarly define $x_i', y_i', m_i', n_i'$ for integers $i' \notin \II{1, k}$.
		 With these definitions, suppose that $x_i \le x_{i+s}', y_i \le y_{i+s}', n_i \le n_{i+s}', m_i \le m_{i+s}'$ for all $i \in \Z$.
		 
		 Then if $\pi, \pi'$ are the rightmost optimizers from $\bp$ to $\bq$ and $\bp'$ to $\bq'$, we have $\pi_i \le \pi_{i+s}'$ for all $i \in \II{1, k} \cap \II{1 - s, k - s}$.
	\end{enumerate}
\end{prop}

Part (ii) above is quite general. The concrete case when $s = 0$ and $k = k'$ is easier to understand. In this case, it amounts to a simple monotonicity of rightmost optimizers as we shift endpoints left and right, or up and down.
	
	\subsection{Pitman transforms}
	
	Recall from the introduction that $\cC^n_0$ is the space of $n$-tuples of continuous functions $f = (f_1, \dots, f_n), f_i:[0,\infty) \to \R$ with $f(0) = 0$. A fruitful way of studying last passage percolation across environments in $\cC^n_0$ is by sorting the environments two lines at a time using a series of two-line Pitman transforms, an approach introduced in \cite{biane2005littelmann}. The Pitman transform is simply the map $W:\cC^2_0 \to \cC^2_0$ from \eqref{E:Wf-definition} for $2$ lines. The following proposition records important properties of this map. 

\begin{prop}
	\label{P:basic-prop-W}
	Let $W:\cC^2_0 \to \cC^2_0$ and $f \in \cC^2_0$. We have
	\begin{enumerate}[nosep, label=(\roman*)]
		\item $Wf_2 \le f_2 \le Wf_1$ and $Wf_2 \le f_1 \le Wf_1$.
		\item For any endpoints $(\bx, 2), (\by, 1)$ with $0 \le x_1$, we have the \textbf{isometry}
		$$
		Wf [(\bx, 2) \to (\by, 1)] = f [(\bx, 2) \to (\by, 1)].
		$$
	\end{enumerate}
\end{prop}

Items (i) is immediate from the definition, and (ii) is \cite[Lemma 4.3]{DOV}.  
We can apply Pitman transforms two lines at a time to understand last passage percolation in $\cC^n_0$. 
For $n \in \N$, an adjacent transposition $\sig_i := (i, i+1) \in S_n$ and $f \in \cC^{\II{1, n}}$, define
\begin{equation}
\label{E:sigma-def}
W_{\sig_i} f = (f_1, \dots, f_{i-1}, W(f_i, f_{i+1}), f_{i+2}, \dots, f_n).
\end{equation}
For more general permutations $\tau \in S_n$, we define
\begin{equation}
\label{E:W-tau}
W_\tau = W_{\sig_{i_1}} \cdots W_{\sig_{i_k}}
\end{equation}
where $\sig_{i_1} \cdots \sig_{i_k} = \tau$ is a \textbf{reduced decomposition} of $\tau$, i.e. a minimal length decomposition of $\tau$ as a product of adjacent transpositions. The right side of \eqref{E:W-tau} is the same for any reduced decomposition, see the discussion preceding Proposition $2.8$ in \cite{biane2005littelmann}. The $n$-line map $W$ in \eqref{E:Wf-definition} is equal to $W_{\rev_n}$, where $\rev_n = n \cdots 1$ is the reverse permutation, see \cite{biane2005littelmann} or Section 3 in \cite{DNV2}.

In \cite{DOV}, the isometry in Proposition \ref{P:basic-prop-W}(ii) was extended to yield an isometry for  $W = W_{\rev_n}$ by using the metric composition law, Proposition \ref{P:MC-law}. We will also extend Proposition \ref{P:basic-prop-W}(ii) to get an isometric property for general permutations; the proof is essentially identical to the proof of \cite[Proposition 4.1]{DOV}.

\begin{prop}
\label{P:Wsig-extension}
Let $\tau \in S_n$ be a permutation and suppose that for some interval $\II{a, b} \sset \II{1, n}$, that $\tau$ is the identity on $\II{a, b}^c$. Then for any endpoints $\bp, \bq$ with $p_i \in \II{n, b} \X [0, \infty)$ and $q_i \in \II{a, 1} \X [0, \infty)$, we have
$$
W_\tau f [\bp \to \bq] = f[\bp \to \bq].
$$ 
\end{prop}

\begin{proof}
By \eqref{E:W-tau}, it suffices to prove the proposition when $\tau$ is an adjacent transposition $\sig_j$ and $\II{a, b} = \{j, j + 1\}$. First assume $j \ne 1, n-1$. For any endpoints $\bp, \bq$ of size $k$ with $p_i \in \II{n, j+1} \X [0, \infty)$ and $q_i \in \II{j, 1} \X [0, \infty)$ for all $i$, Proposition \ref{P:MC-law} ensures that
\begin{equation}
\label{E:f-p-q}
f[\bp \to \bq] = \max_{\bz, \bz' \in \R^k_\le} f[\bp \to (\bz, j+2)] + f[(\bz, j+1) \to (\bz', j)] + f[(\bz', j-1) \to \bq].
\end{equation}
Under the maximum in \eqref{E:f-p-q}, the first and third terms are unchanged when we apply $W_{\sig_j}$ since $W_{\sig_j} f_k = f_k$ for $k \ne j +1, j$. The middle term is unchanged by Proposition \ref{P:basic-prop-W}(ii). Hence the left side of \eqref{E:f-p-q} is unchanged when we apply $W_{\sig_j}$. The cases when $j = 1, n-1$ are similar, except there will be fewer terms on the right side of \eqref{E:f-p-q}.
\end{proof}

	Next, for $j \le i$ define the permutation
	$
	\tau_{i, j} = \sig_{j} \cdots \sig_{i-1}.
	$
	We use the convention that $\tau_{i, i} = \id_n$. The maps $W_{\tau_{i, j}}$ are related to last passage by the following lemma, Lemma 3.10 from \cite{DNV2}. 
	
	\begin{lemma}
		\label{L:W-lemma}
		Let $j \le i \in \II{1, n}$ and $f \in \cC^n_0$. For all $y \ge 0$, we have
		\begin{equation}
		\label{E:f0iyj}
		f[(0, i) \to (y, j)] = W_{\tau_{i, j}} f_j(y).
		\end{equation}
	\end{lemma}


The next proposition builds on Lemma \ref{L:W-lemma}. Recall that for a finite set $J = \{j_1 < \dots < j_k\} \sset \Z$ and $x \in \R$, we use the shorthand
	$
	(x, J) = ((x, j_1), \dots, (x, j_k)).
	$
	\begin{prop}
	\label{P:top-lines}
	Let $I = \{i_1 < \dots < i_k\} \sset \II{1, n}, f \in \cC^n_0,$ set $\tau_I = \tau_{i_k, k} \cdots \tau_{i_1, 1}$, and define $m:I \to \II{1, n}$ by $m(i_j) = j$ for all $j$. For any nonempty subset $J \sset I$ and any vector $\by \in \R^{|J|}_\le$ with $0 \le y_1$, we have
	\begin{equation}
	\label{E:WtauI-isometry}
	W_{\tau_I} f [(0, m(J)) \to (\by, 1)] = f[(0, J) \to (\by, 1)].
	\end{equation}
	We can also explicitly describe of the functions $W_{\tau_I} f_1, \dots, W_{\tau_I} f_k$ as follows:
	$$
	\sum_{i=1}^\ell W_{\tau_I} f_i(y) = f[(0,I^\ell) \to (y^\ell, 1)],
	$$
	for all $0 \le y, \ell \in \II{1, k}$, where $I^\ell = \{i_1, \dots, i_\ell\}$. 
	\end{prop}

	\begin{proof} 
		For each $\ell \in \II{0, k}$, define a map $m_\ell:I \to \II{1, n}$ by setting $m_\ell(i_j) = j$ for $j \le \ell$, and $m_\ell(i_j) = i_j$ otherwise.
		Let $\tau_\ell = \tau_{i_{\ell}, \ell} \cdots \tau_{i_1, 1}$. We will inductively prove the stronger claim that for every $\ell \in \II{0, k}$, for any $J, \by$ as in the proposition, we have
		\begin{equation}
		\label{E:W-tauell}
		W_{\tau_\ell} f [(0, m_\ell(J)) \to (\by, 1)] = f[(0, J) \to (\by, 1)].
		\end{equation}
		The base case when $\ell = 0$ is trivially true. Now suppose that the claim holds at $\ell-1$. It is enough to show that for any $g \in \cC^n_0$ and $J, \by$ as in the proposition, we have 
		\begin{equation}
		\label{E:W-tauell-2}
		W_{\tau_{i_\ell, \ell}} g [(0, m_\ell(J)) \to (\by, 1)] = g [(0, m_{\ell-1}(J)) \to (\by, 1)].
		\end{equation}
		Indeed, \eqref{E:W-tauell-2} implies \eqref{E:W-tauell} by taking $g = W_{\tau_{\ell-1}} f$ and applying the inductive hypothesis. First, by Proposition \ref{P:Wsig-extension}, \eqref{E:W-tauell-2} holds with $m_\ell(J)$ replaced by $m_{\ell-1}(J)$. This immediately implies \eqref{E:W-tauell-2} if $i_\ell \notin J$, so from now on we may assume $i_\ell \in J$. Now,  
		\begin{equation}
		\label{E:W-tauell-3}
		W_{\tau_{i_\ell, \ell}} g [(0, m_\ell(J)) \to (\by, 1)] \le W_{\tau_{i_\ell, \ell}} g [(0, m_{\ell-1}(J)) \to (\by, 1)],
		\end{equation}
		since all disjoint $k$-tuples from $(0, m_\ell(J))$ to $(\by, 1)$ are also disjoint $k$-tuples from $(0, m_{\ell-1}(J))$ because $m_\ell \le m_{\ell - 1}$. For the opposite inequality, consider any disjoint $k$-tuple $\pi$ from $(0, m_{\ell-1}(J))$ to $(\by, 1)$. Let $\pi_\ell$ be the path starting at $(0, i_\ell)$, and let $y = \sup \{t \ge 0 : \pi(t) \ge \ell\}$, where we take $y = 0$ if this set is empty. Then 
		$$
		\|\pi_\ell|_{[0, y]}\|_{W_{\tau_{i_\ell, \ell}} g} \le W_{\tau_{i_\ell, \ell}} g[(i_\ell, 0) \to (y, \ell)] = g[(i_\ell, 0) \to (y, \ell)] = W_{\tau_{i_\ell, \ell}} g_{\ell}(y).
		$$
		The first equality follows from the isometry in Proposition \ref{P:Wsig-extension}, and the second equality follows from Lemma \ref{L:W-lemma}. Therefore if we define a new path $\rho = \ell \wedge \pi_\ell$, we have
		$$
		\|\pi_\ell\|_{W_{\tau_{i_\ell, \ell}} g} \le \|\rho\|_{W_{\tau_{i_\ell, \ell}} g}.
		$$
		Moreover, replacing the path $\pi_\ell$ with $\rho$ in the $k$-tuple $\pi$ yields a new disjoint $k$-tuple $\pi_*$ from $(0, m_\ell(J))$ to $(\by, 1)$ since all paths starting above $\pi_\ell$ started at lines in $\II{1, \ell-1}$. This gives the opposite inequality in \eqref{E:W-tauell-3}, yielding \eqref{E:W-tauell-2}.
	\end{proof}

\begin{remark}
	\label{R:melons-at-other-times}
	Moving forward, we will also want to apply Pitman transforms to more general environments $f$, opened up at times other than $0$. Let $f = (f_1, \dots, f_n), f_i:\R \to \R$. Letting $T_a f(x) = f(x + a) - f(a)$, for $\tau \in S_n$, define
	\begin{equation}
	\label{E:W-shifted}
	W_{a, \tau} f(x) = W_\tau T_a f (x - a)
	\end{equation}
	so that $W_{a, \tau} f_i:[a, \infty) \to \R$ for all $i \in \II{1, n}$.
\end{remark}
\subsection{Brownian motion and the Pitman transform}

When $B \in \cC^2_0$ is a Brownian motion, Pitman's $2M-X$ theorem \cite{pitman1975one} identifies the law of $WB$.

\begin{theorem}
	\label{T:Pitman-Brownian}
	Let $B = (B_1, B_2) \in \cC^2_0$ be two independent standard Brownian motions. Then $(WB_1 + WB_2, WB_1 - WB_2) \eqd (\sqrt{2} R, \sqrt{2} B)$, where $R$ is a Bessel-$3$ process, $B$ is a standard Brownian motion, and the two objects are independent.
\end{theorem}

Theorem \ref{T:Pitman-Brownian} implies that for any interval $[a, b] \sset (0, \infty)$, we have $WB|_{[a, b]} \ll B|_{[a, b]}$. We will need a version of this result for iterated Pitman transforms $W_{\sig}$. This will require the following strengthening of this absolute continuity observation. 

\begin{prop}
	\label{P:brown-abs}
	Let $B \in \cC^n_0$ be a sequence of $n$ independent Brownian motions. For any $\tau \in S_n$ and any $[a, b] \sset (0, \infty)$, we have $W_\tau B|_{[a, b]} \ll B|_{[a, b]}$.
\end{prop}

 We use the following lemma to help with the inductive step.

\begin{lemma}[Lemma 4.5, \cite{ganguly2021local}]
	\label{L:abs-criterion}
	Let $X \in \cC^2_0$ be a random function such that for every $0 < \ep < T$, we have $X|_{[\ep, T]} \ll B|_{[\ep, T]}$ where $B \in \cC^2_0$ is a standard Brownian motion. Suppose also that a.s.,
	\begin{equation}
	\label{E:max-xep}
	\max_{x \in [0, \ep]} X_2(x) - X_1(x) > 0
	\end{equation}
	for all $\ep > 0$. Then for every $0 < \ep < T$, we have $WX|_{[\ep, T]} \ll B|_{[\ep, T]}$.
\end{lemma}

Note that in \cite{ganguly2021local}, Ganguly and Hegde prove a version of Proposition \ref{P:brown-abs} for particular permutations $\sig$. Just as in our case, their key input is Lemma \ref{L:abs-criterion}. 

\begin{proof}
We prove the proposition by induction on $n$. The $n=2$ base case follows from Theorem \ref{T:Pitman-Brownian}, as discussed above. Now suppose the lemma holds at $n-1$, and consider $\tau \in S_n$. Let $m = \tau(n)$. Using the notation $\tau_{i, j}$ introduced prior to Proposition \ref{P:top-lines}, we can write $\tau = \tau_{n, m} \rho$ for some $\rho \in S_n$ with $\rho(n) = n$. By the inductive hypothesis and the fact that $(W_\rho B)_n = B_n$, the proposition holds if $m = n$. We now induct backwards on $m \in \II{1, n}$. Let $m \le n-1$ and suppose the proposition holds for $m + 1$. By Lemma \ref{L:abs-criterion}, the proposition also holds if the set of $x$ where
\begin{equation}
\label{E:Wtau-nm}
W_{\tau_{n, m+1} \rho} B_{m+1}(x) - W_{\tau_{n, m+1} \rho} B_m(x) > 0
\end{equation}
has a limit point at $0$ almost surely. Now, $W_{\tau_{n, m+1} \rho} B_m = W_{\rho} B_m$, and 
$$
W_{\rho} B_m(x) \le W_{\rho} B[(0, n-1) \to (x, 1)] = B[(0, n-1) \to (x, 1)],
$$
where the equality follows from Proposition \ref{P:Wsig-extension}. Moreover, repeated applications of Proposition \ref{P:basic-prop-W}(i) imply that $W_{\tau_{n, m+1} \rho} B_{m+1}(x) \ge W_{\rho} B_{n}(x) = B_n(x)$. Therefore to show \eqref{E:Wtau-nm}, we just need to show that the set of $x$ where 
\begin{equation}
\label{E:Xnn}
B_n(x) - B[(0, n-1) \to (x, 1)] > 0
\end{equation}
has a limit point at $0$ almost surely. Let $I_m$ denote the indicator of the event in \eqref{E:Xnn} for $x = 1/2^m$. By Blumenthal's $0-1$ law, the process $\{I_m, m \in \N\}$ is stationary and ergodic. Therefore by the ergodic theorem, \eqref{E:Xnn} occurs a.s.\ for infinitely many $x = 1/2^m, m \in \N$ as long as $\E I_m > 0$. The fact that $\E I_m > 0$ follows since $B_n(x), B[(0, n-1) \to (x, 1)]$ are independent and $B_n(x)$ has an unbounded upper tail.
\end{proof}

In order to prove Theorem \ref{T:KPZ-fixed-point}, we will also need to understand how two-line Brownian last passage percolation evolves from a vertical initial condition.

\begin{prop}
	\label{P:pre-kpzfp}
	Let $B \in \cC^2_0$ and suppose $B \ll B'$, where $B' \in \cC^2_0$ is a standard Brownian motion. Let $a_1 > a_2$, and define $M_i:[0, \infty) \to \R$ by $M_i(y) = a_i + B[(0, i) \to (y, 1)]$. Set $H = \max(M_1, M_2)$.
	Then a.s. there exists $\tau \in (0, \infty)$ such that
	\begin{equation}
	\label{E:hAA}
	H(x) = \mathbf{1}(x < \tau) M_1(x) + \mathbf{1}(x \ge \tau) M_2(x).
	\end{equation}
	Next, let $X$ be any random variable on $\R$ with positive Lebesgue density everywhere, let $B^*$ be a two-sided standard Brownian motion, let $R$ be a two-sided Bessel-$3$ process, and suppose that all three objects are independent. Then
	$$
	(\tau, \sqrt{2} (H - H(\tau))) \ll (X, [B^* + R](- X + \cdot)|_{[0, \infty)}).
	$$ 
\end{prop}

When $B = B'$, a more precise description of the law of $(\tau, \sqrt{2} (H - H(\tau)))$ falls out of the proof. To prove Proposition \ref{P:pre-kpzfp} we need a lemma.

\begin{lemma}
	\label{L:mutually-abs}
Let $B:[0, t] \to \R$ be a standard Brownian motion, let $R$ be a Bessel-$3$ process on $[0, t]$ started at $0$ and conditioned to end at a location $a$, and let $R'$ be an unconditioned Bessel-$3$ process on $[0, t]$. Suppose all objects are independent. Then $(B + R)|_{[0, t]}$ and $(B + R')|_{[0, t]}$ are mutually absolutely continuous. 
\end{lemma}

\begin{proof}
	It is enough to show this when $R'$ is an independent Bessel-$3$ process on $[0, t]$ conditioned to end at a specific location $a' \ne a$. First, the processes $(B + R')|_{[0, t/2]}$ and $(B + R)_{[0, t/2]}$ are mutually absolutely continuous since $R, R'$ are mutually absolutely continuous when restricted to $[0, t/2]$.

	Next, let $\cF$ be the $\sig$-algebra generated by $(B + R')|_{[0, t/2]}, (B + R)_{[0, t/2]}$. 
	The conditional distribution $\mu_a$ of $a - R(t/2)$ given $\cF$ a.s. has a positive Lebesgue density on all of $\R$. The same is true of the conditional distribution $\mu_{a'}$ of $a' - R(t/2)$ given $B' + R'$ on $[0, t/2]$, so $\mu_{a}, \mu_{a'}$ are mutually absolutely continuous a.s.
	
	The conditional distributions of $(B + R)|_{[t/2, t]} - (B + R)(t/2), (B + R')|_{[t/2, t]} - (B + R')(t/2)$ only depend on $\cF, a, a'$ through $\mu_a, \mu_{a'}$ respectively. Therefore $(B + R')|_{[0, t]}$ and $(B + R)_{[0, t]}$ are mutually absolutely continuous since $\mu_{a}, \mu_{a'}$ are mutually absolutely continuous a.s.
\end{proof}

\begin{proof}[Proof of Proposition \ref{P:pre-kpzfp}]
It is enough to prove this when $(B_1, B_2) = (B_1', B_2')$. Let $a = a_1 - a_2, C_1 = B_1 - B_2, C_2 = B_1 + B_2$. We have that
$$
\tau = \inf \{t \in (0, \infty) : C_1(t) = - a\}.
$$ 
Hence $\tau \ll X$ since $C_1$ is a Brownian motion, and $\tau$ is a stopping time with respect to $C_1$'s filtration. In particular, given $\tau$, the three processes
\begin{equation}
\label{E:tau-x}
\frac{C_2(\tau - x) - C_2(\tau)}{\sqrt{2}}, \quad x \in [0, \tau]; \quad B_1(\tau+x) - B_1(\tau), \quad B_2(\tau+x) - B_2(\tau), \quad x \in [0, \infty) 
\end{equation}
are all independent standard Brownian motions. Moreover, given $\tau$ the process
$$
\frac{C_1(\tau - x) - C_1(\tau)}{\sqrt{2}}, x \in [0, \tau]
$$
is a standard Brownian motion conditioned to stay positive and go from $(0, 0)$ to $(\tau, a)$, independent of all the processes in \eqref{E:tau-x}. In other words, it is a Bessel-$3$ process on $[0, \tau]$ conditioned to end at $a$. Therefore by Lemma \ref{L:mutually-abs},
\begin{equation}
\label{E:tau-abs}
(\tau, 2[M_1(\tau - x) - M_1(\tau)]) = (\tau, [C_1 + C_2](\tau - x) - [C_1 + C_2](\tau) ) \ll (\tau, \sqrt{2}[B^* + R](- x)),
\end{equation}
where the absolute continuity is as functions on $[0, \tau]$. Now, since $\tau$ is a record time for the process $C_1$, for $y \ge \tau$, we have that
$$
M_2(y) - M_2(\tau) = B_2[(\tau, 2) \to (y, 1)]. 
$$
The observation from \eqref{E:tau-x} and Theorem \ref{T:Pitman-Brownian} imply that given $\tau$ and $C_1, C_2$ on $[0, \tau]$, we  have
$$
\sqrt{2} B_2[(\tau, 2) \to (\tau + \cdot, 1)] \eqd B^* + R
$$
as functions on $[0, \infty)$. Combining this with \eqref{E:tau-abs} completes the proof. 
\end{proof}

\section{Isometries in the limit}
\subsection{Preliminaries on universal KPZ limits}
\label{S:universal}

In this section, we gather the necessary background about the Airy line ensemble, the Airy sheet, the directed landscape, and extended versions of these objects.

First, the \textbf{parabolic Airy line ensemble} $\cW = (\cW_i, i \in \N)$ is a random sequence of continuous functions $\cW_i:\R \to \R$. Its lines satisfy $\cW_1 > \cW_2 > \dots$ almost surely. It is defined through a determinantal formula, see \cite{CH} for details. The main fact about the Airy line ensemble that we need is the following. For this next proposition, we let $\cF_{a, b, k}$ denote the $\sig$-algebra generated by
$$
\cW_i(t), \qquad (i, t) \notin \II{1, k} \X (a, b).
$$

\begin{prop}
\label{P:brownian-ac}
Fix $k \in \N, a < b$ and let $B$ be a $k$-dimensional Brownian motion of variance $2$ started from $B(a) = 0$. Then conditionally on $\cF_{a, b, k}$, for any $c \in (a, b)$ we have
\begin{equation}
\label{E:abs-cont-airy}
(\cW_1|_{[a, c]} - \cW_1(a), \dots, \cW_k|_{[a, c]} - \cW_k(a)) \ll (B_1, \dots, B_k)|_{[a, c]}.
\end{equation} 
\end{prop}
Other absolute continuity statements for $\cW$ have appeared earlier (e.g. \cite[Proposition 4.1]{CH}), but these are not phrased to give the conditional absolute continuity we need, so we include a proof. 

\begin{proof}
	The statement is almost immediate from the \textbf{Brownian Gibbs property} for $\cW$, shown in \cite[Theorem 3.1]{CH}. The Brownian Gibbs property says that for any $a < b$ and $k \in \N$, conditional on $\cF_{a, b, k}$, the Airy lines $\cW_1 > \dots > \cW_k$ restricted to the interval $[a, b]$ are given by $k$ independent Brownian bridges of variance $2$ between $\cW_i(a)$ and $\cW_i(b)$ for $i \in \II{1, k}$, conditioned so that $\cW_1(x) > \dots > \cW_k(x) > \cW_{k+1}(x)$ for all $x \in (a, b)$. In particular, for any $c \in (a, b)$ we have that
	\begin{itemize}[nosep]
		\item The conditional law of $(\cW_1(c) - \cW_1(a), \dots, \cW_k(c) - \cW_k(a))$ given $\cF_{a, b, k}$ is absolutely continuous with respect to Lebesgue measure on $\R^k$.
		\item Conditional on $(\cW_1(c) - \cW_1(a), \dots, \cW_k(c) - \cW_k(a))$ and $\cF_{a, b, k}$, the law of 
		$$
		(\cW_1 - \cW_1(a), \dots, \cW_k - \cW_k(a)) |_{[a, c]}
		$$
		is absolutely continuous with respect to the law of $k$ independent Brownian bridges from $0$ at time $a$ to  $\cW_i(c) - \cW_i(a)$ at time $c$ for $i \in \II{1, k}$.
	\end{itemize}
Now, with $B$ as in the statement of the proposition, the law of $B(c)$ is mutually absolutely continuous with respect to Lebesgue measure, and conditional on $B(c)$, the law of $B|_{[a, c]}$ is the law of $k$ independent Brownian bridges from $0$ at time $a$ to $B(c)$ at time $c$. Therefore conditional on $\cF_{a, b, k}$ we have \eqref{E:abs-cont-airy}.
\end{proof}

\begin{definition}
	\label{D:Airy-sheet}
Let $\cW$ be a parabolic Airy line ensemble, and let $\tilde \cW(x) = \cW(-x)$. For $(x, y, z) \in \Q^+ \X \Q^2$ let
\begin{equation}
\label{E:S'-form}
\begin{split}
\cS'(x, y, z) &= \lim_{k \to \infty} \cW[(-\sqrt{k/(2x)}, k) \to (y, 1)]- \cW[(-\sqrt{k/(2x)}, k) \to (z, 1)],\\
\cS'(-x, y, z) &= \lim_{k \to \infty} \tilde \cW[(-\sqrt{k/(2x)}, k) \to (-y, 1)]- \tilde \cW[(-\sqrt{k/(2x)}, k) \to (-z, 1)].
\end{split}
\end{equation}
For $(x, y) \in \Q \setminus \{0\} \X \Q$, define
\begin{equation}
\label{E:Sxy-lim}
\cS(x, y) = \E\cA_1(0)+\lim_{n \to \infty} \frac{1}{2n} \sum_{z=-n}^{n}\cS'(x, y, z) - (z - x)^2,
\end{equation}
The {\bf Airy sheet} defined from $\cW$ is the unique continuous extension of $\cS$ to $\mathbb R^2$. 
\end{definition}

It is not at all clear that the limits in Definition \ref{D:Airy-sheet} exist, or that the resulting function has a continuous extension to $\R^2$. However, almost surely, these properties do hold and so the construction above is well-defined. This is shown in \cite[Section 1.11]{dauvergne2021scaling}. While Definition \ref{D:Airy-sheet} is somewhat involved, one should think of it as capturing the idea that 
\begin{equation}
\label{E:sSWxinfty}
``\cS(x,y) = \cW[(x, \infty) \to (y, 1)]",
\end{equation}
which can be thought of as the single-point limit of Proposition \ref{P:Wsig-extension}.
The more complex Definition \ref{D:Airy-sheet} is required to rigorously make sense of the right hand side of \eqref{E:sSWxinfty}, see \cite{DOV}, \cite{dauvergne2021scaling} for more background. 

The extended landscape is built from the extended Airy sheet. To describe this precisely, we need a notion of multi-point last passage percolation across $\cW$.
For $x \in [0, \infty)$ and $y\in \R$, a nonincreasing cadlag function $\pi:(-\infty, y] \to \N$ is a \textbf{parabolic path} from $x$ to $z$ if
$$
\lim_{z \to -\infty} \frac{\pi(y)}{2z^2} =x.
$$
For a parabolic Airy line ensemble $\cW$ with corresponding Airy sheet $\cS$ as in Definition \ref{D:Airy-sheet}, define the \textbf{path length}
\begin{equation}
\label{E:piSW}
\|\pi\|_\cW = \cS(x,y) + \lim_{z \to -\infty} \lf( \|\pi|_{[z,y]}\|_\cW - \cW[(z, \pi(z)) \to (y, 1)]\rg). 
\end{equation}
The limit above always exists and is nonpositive, see \cite[Lemma 5.1]{dauvergne2021disjoint}. The idea of definition \eqref{E:piSW} is that the best possible path length for $\pi$ should be the last passage value in \eqref{E:sSWxinfty}. This accounts for the first term in \eqref{E:piSW}. For fixed $z$, the second term in \eqref{E:piSW} measures the discrepancy between the length of $\pi$ and the optimal length on the compact interval $[z, y]$; taking $z \to -\infty$ then gives the discrepancy between $\|\pi\|_\cW$ and $\cS(x,y)$.

For $(\bx, \by) \in \R^k_\le \X\R^k_\le$ with $x_1 \ge 0$, we can then define the multi-point last passage value
\begin{equation}
\label{E:Bxy}
\cW[\bx \to \by] = \sup_{\pi} \sum_{i=1}^k \|\pi_i\|_\cW,
\end{equation}
where the supremum is over \textbf{disjoint $k$-tuples} of parabolic paths  $\pi = (\pi_1, \dots, \pi_k)$ from $x_i$ to $y_i$ satisfying $\pi_i(t) > \pi_j(t)$ for all $i < j, t < y_i$. A $k$-tuple achieving \eqref{E:Bxy} is a \textbf{(disjoint) optimizer} from $\bx$ to $\by$, or a \textbf{geodesic} from $x$ to $y$ when $k = 1$.
 In \cite[Theorem 1.3]{dauvergne2021disjoint}, equation \eqref{E:Bxy} was used along with a translation invariance property to characterize the extended Airy sheet. We can give a stronger characterization here, thanks to the recently established landscape symmetry
\begin{equation}
\label{E:sLxsyt}
\cL(x, s; y, t) \eqd \cL(-x, s; -y, t)
\end{equation}
shown in \cite[Proposition 1.23]{dauvergne2021scaling}. The equality \eqref{E:sLxsyt} is as random continuous functions on $\Rd$.

\begin{theorem}
	\label{T:extended-sheet-char}
Let $\cL$ be a directed landscape. Similarly to \eqref{E:WL-definition}, define 
\begin{equation}
\label{E:sW-xk}
\begin{split}
\sum_{i=1}^k W_x \cL_i(y) &= \cL(x^k, 0; (y-x)^k, 1)
\end{split}
\end{equation}
Then for all $x \in \R$, the process $W_x \cL$ is a parabolic Airy line ensemble and a.s. we have that
\begin{equation}
\label{E:isom-only-on-Airy}
\begin{split}
\sum_{i=1}^k W_x \cL_i(y) &= W_0 \cL[x^k \to (y-x)^k], \qquad \text{ for all } x \ge 0, \quad \mathand \\
\sum_{i=1}^k W_x \cL_i(y) &=  \tilde W_0 \cL[(-x)^k \to (x-y)^k], \qquad \text{ for all } x \le 0.
\end{split}
\end{equation}
Here similarly to Definition \ref{D:Airy-sheet}, $\tilde W_0 \cL(y) := W_0 \cL(-y)$. Moreover, define the \textbf{extended Airy sheet} $\cS(\bx, \by) = \cL(\bx, 0; \by, 1).$ This is a real-valued continuous function with domain $\fX := \bigcup_{k \ge 1} \R^k_\le \X \R^k_\le$. Almost surely, $\cS$ satisfies
\begin{equation}
\label{E:isom}
\begin{split}
\cS(\bx, \by) &= W_z \cL[\bx - z^k \to \by - z^k], \qquad x_1 \ge z, \\
\cS(\bx, \by) &= \tilde W_z \cL[z^k - \bx \to z^k - \by], \qquad x_k \le z
\end{split}
\end{equation}
for all $(\bx, \by) \in \fX, z \in \Q$ satisfying the above constraints. In particular, combining \eqref{E:isom-only-on-Airy} and \eqref{E:isom} implies that $\cS$ is a function of $W_0\cL$ and that $\cS|_{\R^2}$ and $W_0\cL$ satisfy the same relationship as $\cS, \cW$ in Definition \ref{D:Airy-sheet}.
\end{theorem}

The equations in \eqref{E:isom} should hold not just for $z \in \Q$, but simultaneously for all $z$. We do not pursue this technical extension here. Note that to make sense of the points $-\bx, -\by \in \R_\le^k$ in \eqref{E:isom} we need to reverse the order of the coordinates in $\bx, \by$. From now on we ignore this minor point.
\begin{proof}
	This really just consists of gathering all the necessary results from \cite{dauvergne2021disjoint}.
The first equality in \eqref{E:isom} for $x = 0$ and $z = 0$ is one of the main results of \cite{dauvergne2021disjoint}. More precisely, it follows by combining Equation (7), Theorem 1.3, and Theorem 1.6 from that paper. This combination also immediately implies \eqref{E:isom-only-on-Airy} for $x \ge 0$. Now, we have the symmetry
\begin{equation}
\label{E:translation-L}
\cL(x, s; y, t) \eqd \cL(x + r, s; y + r, t),
\end{equation}
where the equality is as random functions in $\Rd$, see \cite[Lemma 10.2.2]{DOV}. This implies that the first equality in \eqref{E:isom} also holds almost surely for any $z \in \Q$. The symmetry \eqref{E:sLxsyt} then implies all the `tilde' equalities in both \eqref{E:isom-only-on-Airy} and \eqref{E:isom}. 
\end{proof}

To understand the extended Airy sheet, we study geodesics, path length and disjoint optimizers in $\cW$. We record the following basic properties of these objects, from \cite{dauvergne2021disjoint}. For this proposition, a $k$-tuple of parabolic paths $\pi$ from $\bx$ to $\by$ is \textbf{locally optimal} if for any $t < y_1$, the $k$-tuple $(\pi_1|_{[t, y_1]}, \dots \pi_k|_{[t, y_k]})$ is a disjoint optimizer from $(t, \pi(t))$ to $(\by, 1)$. 

\begin{prop}
	\label{P:geodesic-properties}
	Let $\cW$ be an Airy line ensemble. We have the following properties.
	\begin{enumerate}[nosep, label=(\roman*)]
		\item (Lemma 4.2(i), \cite{dauvergne2021disjoint}) For any fixed $(x, y) \in [0, \infty) \X \R$, there exists a unique geodesic in $\cW$ from $x$ to $y$ a.s. We call this geodesic $\pi\{x, y\}$.
		\item (Lemma 4.2(iv), \cite{dauvergne2021disjoint}) For any fixed $0 \le x < x'$ and $y \in \R$, we have
		$$
		\lim_{r \to \infty} \p(\pi\{x, y\}(z) < \pi\{x', y + r\}(z) \text{ for all } z < y) = 1.
		$$
		\item (Lemma 4.3, \cite{dauvergne2021disjoint}) The following statement holds almost surely.
		Let $\pi_1,\pi_2$ be any two parabolic paths across $\cW$ from any point $x$ to any points $z_1, z_2$ respectively, such that for some $z_0<z_1\wedge z_2$, we have $\pi_1(y)=\pi_2(y)$ for any $y\le z_0$.
		Then
		\[
		\|\pi_1\|_\cW - \|\pi_1|_{[z_0, z_1]}\|_\cW
		=
		\|\pi_2\|_\cW - \|\pi_2|_{[z_0, z_2]}\|_\cW.
		\]
		Here note that $\|\pi|_{[a, b]}\|_{\cW}$ is a usual path length in the environment $\cW$ defined with a sum of increments as in \eqref{E:pi-length}, not an infinite path length.
		\item (Proposition 4.5(i), \cite{dauvergne2021disjoint}) For parts (iv)-(vi), let $\bx = (x_1 < x_2 < \dots < x_k)$ with $x_1 > 0$, and let $\by \in \R^k_\le$. Suppose that $\pi$ is a disjoint optimizer from $\bx$ to $\by$. Then $\pi$ is locally optimal.
		\item (Proposition 4.5(ii), \cite{dauvergne2021disjoint}) A.s.\ there is a unique optimizer $\pi = (\pi_1, \dots, \pi_k)$ from $\bx$ to $\by$ in $\cW$. Moreover, letting $\pi\{x_i, 0\}$ be the a.s.\ unique geodesic in $\cW$ from $x_i$ to $0$, there exists a random $Y \in \R$ such that  $\pi\{x_i, 0\}(t) = \pi_i(t)$ for all $t \le Y, i \in \II{1, k}$.
		\item (Proposition 4.5(iii), \cite{dauvergne2021disjoint}) A.s., the only $k$-tuple $\pi$ from $\bx$ to $\by$ in $\cW$ which is locally optimal is the unique optimizer from $\bx$ to $\by$.
	\end{enumerate}
\end{prop}

%

We also require a crude bound on the extended Airy sheet, and a simple fact about equality of Airy sheet values.
This next lemma is an immediate consequence of \cite[Lemma 7.7]{dauvergne2021disjoint}, which gives a bound for the whole extended landscape.

\begin{lemma}
	\label{L:tails-Airy-sheet}
For every $\eta > 0$, there exists an $R > 0$ such that
$$
\lf| \cS(\bx, \by) + \|\bx - \by\|_2^2 \rg| \le R(1 + \|\bx\|_1 + \|\by\|_1)^\eta
$$
for all $(\bx, \by) \in \fX$.
\end{lemma}

\begin{lemma}[Lemma 7.10, \cite{dauvergne2021disjoint}]
	\label{L:point-push}
	Let $\cS$ be the extended Airy sheet. Then a.s.\ for any $\bx^{(1)} \le \bx^{(2)} \le \bx^{(3)} \le \bx^{(4)} \in \R^k_\le$, and $\by^{(1)} \le \by^{(2)} \le \by^{(3)} \le \by^{(4)} \in \R^k_\le$, if 
	\[
	\cS((\bx^{(2)}, \bx^{(3)}), (\by^{(2)}, \by^{(3)})) = \cS(\bx^{(2)}, \by^{(2)}) + \cS(\bx^{(3)}, \by^{(3)}),    
	\]
	then
	\[
	\cS((\bx^{(1)}, \bx^{(4)}), (\by^{(1)}, \by^{(4)})) = \cS(\bx^{(1)}, \by^{(1)}) + \cS(\bx^{(4)}, \by^{(4)}).
	\]
\end{lemma}

Finally, we use the following relationship between extended landscape values and geodesic disjointness.
\begin{prop}[Corollary 1.11, \cite{dauvergne2021disjoint}]
	\label{P:disjointness} A.s. the following holds.
	For every $(\bx, s; \by, t) \in \fX_\uparrow$, we have
	\begin{equation*}
	\cL(\bx, s; \by, t) = \sum_{i=1}^k \cL(x_i, s; y_i, t)
	\end{equation*}
	if and only if there exist $\cL$-geodesics $\pi_1, \dots, \pi_k$ where $\pi_i$ goes from $(x_i, s)$ to $(y_i, t)$, satisfying $\pi_i(r) < \pi_{i+1}(r)$ for all $i \in \II{1, k-1}$ and $r \in (s, t)$.
\end{prop}

	\subsection{Proof of Theorem \ref{T:landscape-iso}}
	
	In this section, we prove Theorem \ref{T:landscape-iso}. Before doing so, we digress slightly to give a path-based definition of \eqref{E:W-LPP} that may help the reader with intuition. 
	
	\begin{remark}
		\label{R:path-def-infty}
Let $f = (f_1, \dots, f_k), f_i:\R \to \R$. Consider a nonincreasing cadlag function $\pi:(-\infty, y] \to \II{m, n}$, and assume that $\pi$ stabilizes at line $n$ as $z \to -\infty$. We call $\pi$ a path from $(-\infty, n)$ to $(y, m)$. Matching with \eqref{E:pi-length}, define the path length
\begin{equation*}
\|\pi\|_f = \sum_{i = m}^{n} [f_i(\pi_i) - f_i(\pi_{i+1})].
\end{equation*}
In the above sum, for $i \le n$ we let $\pi_i = \inf \{t \in (-\infty, y] : \pi(t) < i\}$ be the time when $\pi_i$ jumps off of line $i$. We also set $\pi_i = y$ if $\pi(t) \ge i$ for all $i$. 
 In the above sum, the term $f_n(\pi_{n+1})$ needs to be defined separately; we simply set it to $0$. Then for $I = (i_1, \dots, i_\ell)$ and $\bp = ((x_1, m_1), \dots, (x_\ell, m_\ell))$ we can write
\begin{equation}
\label{E:f-sum-pi}
f[(-\infty, I) \to \bp] = \sup_\pi \sum_{i=1}^k \|\pi_i\|_f,
\end{equation}
where the supremum is over all $k$-tuples $\pi = (\pi_1, \dots, \pi_k)$, where $\pi_j$ is a path from  $(-\infty, i_j)$ to $p_i = (y_i, m_i)$, and for all $i$ we have $\pi_i < \pi_{i+1}$ on $(-\infty, x_i)$. The equality between the right sides of \eqref{E:f-sum-pi} and \eqref{E:W-LPP} follows from comparing definitions.
	\end{remark}
	
\begin{proof}[Proof of Theorem \ref{T:landscape-iso}, part $1$]
			The asymptotics in \eqref{E:asymptotics} are immediate from Lemma \ref{L:tails-Airy-sheet}. The stability in \eqref{E:WxLx-stable} then follows from \eqref{E:asymptotics} and the fact that $x_1 < x_2 < \dots < x_k$.
		\end{proof}
	
To prove Theorem \ref{T:landscape-iso}.2, we need a few lemmas. The first is a consequence of Proposition \ref{P:geodesic-properties}. Here and throughout the section, we let $\cS, \cL, \cW := W_0 \cL$ be coupled as in Theorem \ref{T:extended-sheet-char}.

	\begin{lemma}
		\label{L:S-lemma}
		Let $\bx \in \R^k_\le$ satisfy $0 < x_1 < \dots < x_k$, let $[a, b] \sset \R$, and let $\pi = (\pi_1, \dots, \pi_k)$, where each $\pi_i$ is the (a.s.\ unique) $\cW$-geodesic from $x_i$ to $0$. Then there exists a random integer $Y < a$ such that for all $z \le Y, I \sset \II{1, k}$ and $\by \in [a, b]^{|I|}_\le$, we have
		\begin{equation}\label{E:S-diff}
		\cS(\bx^I, \by) = \cW[(z, [\pi(z)]^{I}) \to (\by, 1)] + \sum_{i \in I} \| \pi_i|_{(-\infty, z]}\|_\cW.
		\end{equation}
		Here $\| \pi_i|_{(-\infty, z]}\|_\cW$ is a parabolic path length and recall that $\bx^I$ denotes the vector in $\R^{|I|}$ whose coordinates are $x_i, i \in I$.
	\end{lemma}
	
	\begin{proof}
		First, since there are only finitely many choices for $I$, it suffices to prove that for each $I$ there exists $Y = Y(I)$ such that \eqref{E:S-diff} holds whenever $z \le Y, \by \in [a, b]^{|I|}_\le$. Without loss of generality, we may then assume $Y = \II{1, k}$.
		
		By Proposition \ref{P:geodesic-properties}(v), the a.s.\ unique optimizers $\tau, \tau'$ from $\bx$ to $a^k$ and $\bx$ to $b^k$ satisfy $\tau(z) = \tau'(z) = \pi(z)$ for all $z \le Y$ for some random integer $Y < a$. Next, for any $\by \in [a, b]^k_\le$ and $z \le Y$, let $\rho_{z, \by}$ be the rightmost optimizer from $(z, \tau(z))$ to $(\by, 1)$ in $\cW$. By the monotonicity in Proposition \ref{P:basic-optimizers}(ii), we have $\tau \le \rho_{z, \by} \le \tau'$ and so
		\begin{equation}
		\label{E:tauzY}
		\tau|_{[z, Y]} = \rho_{z, \by}|_{[z, Y]} = \tau'|_{[z, Y]}.
		\end{equation}
		Therefore $\rho_{z, \by}|_{[z \vee z', \infty)} = \rho_{z', \by}|_{[z \vee z', \infty)}$ for all $z, z' \le Y$. Hence, there is a $k$-tuple $\rho$ from $\bx$ to $\by$ such that $\rho(t) = \rho_{z, \by}(t)$ for all $z \le Y, t \ge z$. By construction $\rho$ is a disjoint $k$-tuple from $\bx$ to $\by$ which is locally optimal since each $\rho_{z, \by}$ is an optimizer. Therefore by Proposition \ref{P:geodesic-properties}(vi), $\rho$ is a.s.\ an optimizer from $\bx \to \by$. The first relationship \eqref{E:isom} with $z = 0$ then implies that almost surely,
		$$
		\cS(\bx, \by) = \cW[\bx \to \by] = \sum_{i=1}^k \|\rho_i\|_\cW.
		$$
		Next, for any $z \le Y$, since $\rho|_{(-\infty, Y]} = \pi|_{(-\infty, Y]}$ by Proposition \ref{P:geodesic-properties}(iii) we can write 
		\begin{equation*}
		\sum_{i=1}^k \|\rho_i\|_\cW = \sum_{i=1}^k \|\rho_i|_{[z, y_i]}\|_\cW - \sum_{i=1}^k \|\pi_i|_{[z, 0]}\|_\cW + \sum_{i=1}^k \|\pi_i\|_\cW.
		\end{equation*}
		The first term on the right side above equals $\cW[(z, \pi(z)) \to (\by, 1)]$ by construction. Also, by Proposition \ref{P:geodesic-properties}(iii) again, for all $i$ we have $\|\pi_i\|_\cW - \|\pi_i|_{[z, 0]}\|_\cW =\|\pi_i|_{(-\infty, z]}\|_\cW$. This yields \eqref{E:S-diff} a.s. for any fixed $\by \in [a, b]^k_\le$. Continuity of both sides in $\by$ gives \eqref{E:S-diff} simultaneously for all $\by \in [a, b]^k_\le$.
	\end{proof}

The next lemma gives a simple measurability property for path length; it is a variant of the stronger \cite[Lemma 4.3]{dauvergne2021disjoint}.

\begin{lemma}
	\label{L:mble}
	Fix $a \in \R,$ and let $\Pi$ be the set of parabolic paths from $x$ to $y$ for some $y \le a, x \ge 0$. Let $\cF_a$ be the $\sig$-algebra generated by all null sets and by $\cW|_{(-\infty, a]}$, and let $F:\Pi \to \R \cup \{-\infty\}$ be the function recording path length in $\cW$. Then $F$ is $\cF_a$-measurable. 
\end{lemma}

\begin{proof}
	By looking at the definition of parabolic path length in \eqref{E:piSW}, we see that it is enough to prove that $\cS(\cdot, a)|_{[0, \infty)}$ is $\cF_a$-measurable. By continuity of $\cS$ it suffices to prove that $\cS(x, a)$ is $\cF_a$-measurable for all $x \in \Q^+$. Going back to Definition \ref{D:Airy-sheet}, we have that $\cS'(x, n, a)$ is $\cF_a$-measurable for all $n \le a$. Therefore it is enough to show that almost surely
	\begin{equation*}
	\cS(x, a) = \E\cW_1(0)+\lim_{n \to \infty} \frac{1}{a + n} \sum_{m \in [-n, a] \cap \Z} \cS'(x, a, m) - (m-a)^2.
	\end{equation*}
	By Definition \ref{D:Airy-sheet}, $\cS'(x, a, m) = \cS(x, a) - \cS(x, m)$ so the above is equivalent to the equation 
	\begin{equation*}
	\E\cW_1(0) = \lim_{n \to \infty} \frac{1}{a + n} \sum_{m \in [-n, a] \cap \Z} \cS(x, m) + (m-x)^2,
	\end{equation*}
	which follows from the fact that $f(m) = \cS(x, m) + (m-x)^2$ is a stationary Airy process with pointwise mean $\E\cW_1(0)$, and this process is ergodic, see equation (5.15) in \cite{prahofer2002scale}. 
\end{proof}

Finally, we will use a fact about locally Brownian functions, whose proof we leave as a simple exercise for the reader.

\begin{lemma}
\label{L:B-brownian-motion}
Let $t < a < 0  < b$, let $X$ be a random vector in $\R^k$, and let $W:[t, \infty) \to \R^k$ be any random continuous function such that conditional on $X$ we have
$$
W|_{[a, b]} \ll B|_{[a, b]},
$$
where $B$ is a standard $k$-dimensional Brownian motion started from $B(t) = 0$.
Then for any $[a, b] \sset (t, \infty)$ conditional on both $X$ and $W(0)$ we have
$$
W|_{[a, b]} - W(0) \ll B',
$$
where $B'$ is a standard $k$-dimensional Brownian motion started from $B(0) = 0$.
\end{lemma}
	
\begin{proof}[Proof of Theorem \ref{T:landscape-iso}, part $2$]

First, by translation invariance of $\cL$ (Equation \eqref{E:translation-L}), we may assume $x_1 > 0$. Now, as in Lemma \ref{L:S-lemma}, let $\pi = (\pi_1, \dots, \pi_k)$, where $\pi_i$ is the unique $\cW$-geodesic from $x_i$ to $0$. Also, for $z < y_1$, define $X^z \in \R^k$ by
$$
X^z_i := \| \pi_i|_{(-\infty, z]}\|_\cW.
$$
Fix $n \in \N$, and let $Y_n$ be the random integer specified by Lemma \ref{L:S-lemma} for the interval $[-n, n]$. For all $\ell \in \II{1, k}$ and $y \in [-n, n]$ we have
\begin{equation}
\label{E:sumi1}
\sum_{i=1}^\ell W_\bx \cL_i(y) = \cS(\bx^{\II{1, \ell}}, y^\ell) = \cW[(Y_n, [\pi(Y_n)]^{\II{1, \ell}}) \to (y^\ell, 1)] + \sum_{i=1}^\ell X^{Y_n}_i.
\end{equation}
Now, using the notation $\tau_I, W_{z, \tau_I}$ from Proposition \ref{P:top-lines} and Remark \ref{R:melons-at-other-times}, for $z \in \Z$ and $I = \{i_1 < \dots < i_k\} \sset \N$,  we set
$$
\cW^{I, z} := W_{z, \tau_I} (\cW_1, \cW_2, \dots, \cW_{i_k}). 
$$
By Proposition \ref{P:top-lines} and \eqref{E:sumi1} we then have
\begin{equation}
\label{E:general-SW-relation}
W_\bx \cL_i(y) = \cW^{\pi(Y_n), Y_n}_i(y) + X^{Y_n}_, \qquad \mathforall i \in \II{1, k}, y \in [-n, n].
\end{equation}
Since there are only countably many choices for $Y_n, \pi(Y_n)$, to complete the proof it suffices to show that for every fixed $I$ and $z \in \Z$ with $z < -n$, conditional on $\cW^{I, z}(0)$ and $X^z$, on the interval $[-n, n]$ the function
$$
[\cW^{I, z} + X^z]-[\cW^{I, z} + X^z](0) = \cW^{I, z} - \cW^{I, z}(0) 
$$
is absolutely continuous with respect to a $k$-dimensional Brownian motion $B'$ of variance $2$ on the same interval started from $B'(0) = 0$. 

By Lemma \ref{L:mble}, $X^z$ is $\cF_z$-measurable. Therefore by Proposition \ref{P:brownian-ac}, conditional on $X^z$, on the interval $[z, n]$ the sequence 
$$
(\cW_1 - \cW_1(z), \dots, \cW_{i_k} - \cW_{i_k}(z)) 
$$
is absolutely continuous with respect to an $i_k$-dimensional Brownian motion of variance $2$. Hence by Proposition \ref{P:brown-abs}, conditional on $X^z$ we have that 
$
\cW^{I, z}|_{[-n, n]} \ll B|_{[-n, n]},
$
where $B$ is an $i_k$-dimensional Brownian motion of variance $2$ started from $B(z) = 0$. Lemma \ref{L:B-brownian-motion} then yields the result.
\end{proof}

We split the proof of Theorem \ref{T:landscape-iso}.3 up using two lemmas.

\begin{lemma}
\label{L:strong-disjointness}
Fix $\bx \in \R^k_\le$ with $x_1 < \dots < x_k$. A.s., there exists a random vector $\by \in [0, \infty)^k_\le$ such that for every $J \sset \II{1, k}$, we have
\begin{equation}
\label{E:xJyJ-1}
\cS(\bx^{J}, \by^J) = \sum_{i \in J} \cS(x_i, y_i), \qquad \mathand \qquad W_\bx \cL[\bx^J \to \by^J] = \sum_{i \in J} W_\bx \cL[x_i \to y_i].
\end{equation}
Also, for any $y \in \R$ a.s.\ there exists $R > 0$ such that for $r \ge R, \ell \in \II{1, k-1}$ we have
\begin{equation}
\label{E:S-x-R}
\cS(\bx, ((y - r)^\ell, y^{k-\ell})) = \cS(\bx^{\II{1, \ell}}, (y - r)^\ell) + \cS(\bx^{\II{\ell + 1, k}}, y^{k-\ell}).
\end{equation}
\end{lemma}

\begin{proof}
Let $Z_0(0)$ be as in Theorem \ref{T:landscape-iso}.1, 
and let $\tau_i\{y\}$ be the rightmost geodesic in $W_\bx \cL$ from $(Z_0(0), i)$ to $(y, 1)$. Also, let $\pi_i\{y\}$ be the a.s unique geodesic in $\cW$ from $x_i$ to $y$. To prove \eqref{E:xJyJ-1}, it is enough to show that we can find random points $0 = y_1 < \dots < y_k \in \Z$ such that $\tau_i\{y_i\}(z) < \tau_{i+1}\{y_{i+1}\}(z)$ and $\pi_i\{y_i\}(z) < \pi_{i+1}\{y_{i+1}\}(z)$ for all $z < y_i, i \in \II{1, k-1}$. 

We do this by induction on $\ell \in \II{1, k}$. The base case $\ell = 1$ is trivially true. Suppose now that we have constructed $y_1, \dots, y_\ell$. Then we let $y_{\ell + 1}$ be the first integer such that $\tau_\ell\{y_\ell\}(z) < \tau_{\ell+1}\{y_{\ell+1}\}(z)$ and $\pi_\ell\{y_\ell\}(z) < \pi_{\ell+1}\{y_{\ell+1}\}(z)$ for all $z < y_\ell$. We need to justify that such a $y_{\ell + 1}$ exists. This uses the asymptotics in \eqref{E:asymptotics} to establish the inequality for $\tau_\ell, \tau_{\ell + 1}$, and the disjointness estimate in Proposition \ref{P:geodesic-properties}(ii) applied to the points $x_\ell < x_{\ell + 1}$ and $y_\ell$ to establish the inequality for $\pi_\ell, \pi_{\ell + 1}$. Note that Proposition \ref{P:geodesic-properties}(ii) only applies to fixed triples, whereas we are using a random point $y_\ell$. This does not cause any issues since $y_\ell$ can only take on countably many values.

We move to \eqref{E:S-x-R}. By translation invariance of $\cL$, \eqref{E:translation-L}, we may assume $y = 0$. Moreover, by the flip symmetry \eqref{E:sLxsyt}  it is enough to show that there exists $R > 0$ such that for all $r \ge R$ and $\ell \in \II{1, k-1}$ we have
\begin{equation}
\label{E:S-yy1}
\cS(\bx, (0^\ell, r^{k-\ell})) = \cS(\bx^{\II{1, \ell}}, 0^\ell) + \cS(\bx^{\II{\ell + 1, k}}, r^{k-\ell}).
\end{equation} 
Letting $\by = (y_1, \dots, y_k)$ be as constructed above, \eqref{E:xJyJ-1} implies that
\begin{equation}
\label{E:S-yy}
\cS(\bx, \by) = \cS(\bx^{\II{1, \ell}}, \by^{\II{1, \ell}}) + \cS(\bx^{\II{\ell + 1, k}}, \by^{\II{\ell + 1, k}}).
\end{equation} 
Now, letting $R = y_k$ and using that $y_1 = 0$, \eqref{E:S-yy1} follows from \eqref{E:S-yy} and Lemma \ref{L:point-push}.
\end{proof}

\begin{lemma}
\label{L:part-3-lemma}
A.s., for every $J \sset \II{1, k}$ and $\by \in \R^{|J|}_\le$, we have
\begin{equation}
\label{E:SCW}
\begin{split}
\cS(\bx^J, 0^{|J|}) - \cS(\bx^J, \by)
= &W_\bx \cL[(-\infty, J) \to (0^{|J|}, 1)] - W_\bx \cL[(-\infty, J) \to (\by, 1)].
\end{split}
\end{equation}
\end{lemma}

In the proof, we use notation and results contained within the proof of Theorem \ref{T:landscape-iso}.2.

\begin{proof}
	Both sides of \eqref{E:SCW} are continuous in $\by$, so it suffices to prove that the equality holds a.s.\ for fixed $J = \{j_1 < \dots< j_\ell\}$ and $\by \in \R^\ell_\le$. Let $[-m, m]$ be an interval containing all $y_i$, let $R > 0$ be as in Lemma \ref{L:strong-disjointness} with $y = -m$, and let $Z_0(-m) \le -m$ be as in Theorem \ref{T:landscape-iso}.1. Set
	$$
	N = \max (\cl{R}, - Z_0(-m)).
	$$
Letting $Y_n, n \in \N$ and $\pi$ be as in the proof of Theorem \ref{T:landscape-iso}.2, by Lemma \ref{L:S-lemma}, the left side of \eqref{E:SCW} equals
\begin{equation*}
\cW[(Y_N, [\pi(Y_N)]^{J}) \to (0^\ell, 1)] - \cW[(Y_N, [\pi(Y_N)]^{J}) \to (\by, 1)].
\end{equation*}
Using the notation $\cW^{I, z}$ from Theorem \ref{T:landscape-iso}.2, by Proposition \ref{P:top-lines} this equals
	\begin{equation}
	\label{E:Wpi1}
	\cW^{\pi(Y_N), Y_N}[(Y_N, J) \to (0^{\ell}, 1)] - \cW^{\pi(Y_N), Y_N}[(Y_N, J) \to (\by, 1)].
	\end{equation}
	We claim that \eqref{E:Wpi1} equals
	\begin{equation}
	\label{E:Wpi2}
	\cW^{\pi(Y_N), Y_N}[(-N, J) \to (0^{\ell}, 1)] - \cW^{\pi(Y_N), Y_N}[(-N, J) \to (\by, 1)].
	\end{equation}
	For this, it is enough to show that rightmost optimizers in $\cW^{\pi(Y_N), Y_N}$ from $(Y_N, J)$ to  $(0^\ell, 1)$ and $(\by, 1)$ simply have constant paths following the lines in $J$ up to time $-N$. For this, by the monotonicity in Proposition \ref{P:basic-optimizers}(ii), it suffices to prove the same claim for the rightmost optimizer $\pi = (\pi_1, \dots, \pi_\ell)$ from 
	$(Y_N, J)$ to  $((-m)^\ell, 1)$. Again using Proposition \ref{P:basic-optimizers}(ii), we have $\pi_i \ge \tau^i_1$, where $\tau^i = (\tau^i_1, \dots, \tau^i_{k - j_i + 1})$ is the rightmost optimizer in $\cW^{\pi(Y_N), Y_N}$ from $(Y_N, \II{j_i, k})$ to $((-m)^{k - j_i + 1}, 1)$. The paths in $\tau^i$ follow the lines $j_i, \dots, k$ up to time $-N$ if and only if
	\begin{equation*}
	\begin{split}
\cW^{\pi(Y_N), Y_N}[(Y_N, \II{1, j_i - 1}) \to &((-N)^{j_i - 1}, 1)] +\cW^{\pi(Y_N), Y_N}[(Y_N, \II{j_i, k}) \to ((-m)^{k - j_i}, 1)]  \\
&= \cW^{\pi(Y_N), Y_N}[(Y_N, \II{1, k}) \to ((-N)^{j_i - 1}, (-m)^{k - j_i + 1}), 1)].
\end{split} 
	\end{equation*}
	By Proposition \ref{P:top-lines}, this equality is equivalent to an equality of last passage values in $\cW$,
	\begin{equation*}
	\begin{split}
	\cW[(Y_N, [\pi(Y_N)]^{\II{1, j_i - 1}}) \to &((-N)^{j_i - 1}, 1)] +\cW[(Y_N, [\pi(Y_N)]^{\II{j_i, k}}) \to ((-m)^{k - j_i}, 1)]  \\
	&= \cW[(Y_N, \pi(Y_N)) \to ((-N)^{j_i - 1}, (-m)^{k - j_i + 1}), 1)],
	\end{split} 
	\end{equation*}
	which by Lemma \ref{L:S-lemma}, is equivalent to an equality for extended Airy sheet values: 
	\begin{equation*}
	\cS(\bx^{\II{1, j_i - 1}},(-N)^{j_i - 1}) +  \cS(\bx^{\II{j_i, k}},(-m)^{j_i - 1}) = \cS(\bx^{\II{1, k}},((-N)^{j_i - 1},(-m)^{j_i - 1})).
	\end{equation*}
	This equality follows from \eqref{E:S-x-R} since $N \ge R$. This completes the proof that \eqref{E:Wpi1}=\eqref{E:Wpi2}. 
	
	Now, increments of $\cW^{\pi(Y_N), Y_N}$ and $W_\bx \cL$ are the same on the interval $[-N, N]$ by \eqref{E:general-SW-relation}, and since $-N \le Z_0(-m)$, the right side of \eqref{E:SCW} is unchanged if we replace $-\infty$ with $-N$. Therefore \eqref{E:Wpi2} equals the right side of \eqref{E:SCW}, completing the proof.
\end{proof}

\begin{proof}[Proof of Theorem \ref{T:landscape-iso}, part $3$]
	By Lemma \ref{L:part-3-lemma}, for every $J \sset \II{1, k}$, there exists a constant $\al(J)$ such that
$$
\cS(\bx^J, \by) = W_\bx \cL[(-\infty, J) \to (\by, 1)] + \al(J)
$$
for all $\by$. We just need to show that $\al(J) = 0$ for all $J$. First, \eqref{E:xJyJ-1} in Lemma \ref{L:strong-disjointness} implies that 
$$
\al(J) = \sum_{\ell\in J} \al(\ell), \qquad \mathand \qquad \al(\II{1, \ell}) = \al(\II{1, \ell-1}) + \al(\ell)
$$
for all $J, \ell$, so it is enough to show that $\al(\II{1, \ell}) = 0$ for all $\ell \in \II{1, k}$. This follows from the definitions \eqref{E:WLx} and \eqref{E:W-LPP}, which together imply that 
\[
\cS(\bx^{\II{1, \ell}}, 0^\ell) = \sum_{i=1}^\ell W_\bx \cL_i(0) = W_\bx \cL[(-\infty, \II{1, \ell}) \to (0^\ell, 1)]. \qedhere \] 
\end{proof}

\section{Consequences}

In this section we prove all the remaining theorems. We start with Theorems \ref{T:Difference}, \ref{T:geodesic-frame}, and \ref{T:KPZ-fixed-point}, as Theorem \ref{T:main-comparison} is a bit more involved. 

\begin{proof}[Proof of Theorem \ref{T:Difference}] Part $1$ is immediate from parts $1$ and $2$ of Theorem \ref{T:landscape-iso} and the definition \eqref{E:WLx} of $W_\bx \cL$. For part $2$, using part $1$ and the asymptotics in Theorem \ref{T:landscape-iso}.1 there exists a random integer $N < a-1$ such that for all $y \in [a, b]$ we can write
$$
A^{x_1, x_2}(y) = \max_{N \le z \le y} W_\bx \cL_2(z) - W_\bx \cL_1(z).
$$
Therefore we just need to show that for all $n \in - \N, n < a$, that the process $\max_{n \le z \le y} W_\bx \cL_2(z) - W_\bx \cL_1(z), y \in [a, b]$ is absolutely continuous with respect to the running maximum of a Brownian motion $B$ of variance $4$ on $[a, b]$, started at time $a-1$. This follows from local absolute continuity of $W_\bx \cL_2 - W_\bx \cL_1$ with respect to Brownian motion, Theorem \ref{T:landscape-iso}.2. Theorem \ref{T:landscape-iso}.2 also implies that the set of points $y$ where
\begin{equation*}
W_\bx \cL_2(y) - W_\bx \cL_1(y) = \sup_{z \le y} W_\bx \cL_2(z) - W_\bx \cL_1(z) = A^{x_1, x_2}(y)
\end{equation*}
are exactly the support of the associated measure $\mu_{x_1, x_2}$. In other words, there cannot exist a point $y$ where
$$
W_\bx \cL_2(y) - W_\bx \cL_1(y) = \sup_{z \le y'} W_\bx \cL_2(z) - W_\bx \cL_1(z)
$$ 
for all $y'$ in some open neighbourhood containing $y$, since a.s.\ no such $y$ exists for Brownian motion. Therefore by part $1$ of Theorem \ref{T:Difference}, the support of $\mu_{x_1, x_2}$ is the set of times where
$$
\cL((x_1, x_2), 0; z^2, 1) = \cL(x_1, 0; z, 1) + \cL(x_2, 0; z, 1).
$$
By Proposition \ref{P:disjointness}, this is the same as the second set described in Theorem \ref{T:Difference}.$4$.
\end{proof}

\begin{proof}[Proof of Theorem \ref{T:geodesic-frame}] By shift and scale invariance of $\cL$, we can set $s=0, t=1$. For every $x_1, x_2, y_1, y_2 \in \Q$, the $\sig$-algebra $\cF_{0, 1}$ contains the information about whether or not the geodesics from $(x_1, 0)$ to $(y_1, 1)$ and $(x_2, 0)$ to $(y_2, 1)$ are disjoint. Using Proposition \ref{P:disjointness}, this is exactly the set where
$$
\cL((x_1, x_2), 0; (y_1, y_2), 1) = \cL(x_1, 0; y_1, 1) + \cL(x_2, 0; y_2, 1).
$$
By continuity of the extended landscape $\cL$ and Theorem \ref{T:Difference}.$4$, this gives us access to all the supports $S_{x_1, x_2}$ of the measures $\mu_{x_1, x_2}$ with CDFs
$$
A^{x_1, x_2}(z) = \cL(x_2, 0; z, 1) - \cL(x_1, 0; z, 1)
$$ 
for $x_1 < x_2 \in \Q$. Now, fix $x_1 < x_2 \in \Q$ and a compact rational interval $[a, b]$. By Theorem \ref{T:Difference}.1 and the asymptotics in \eqref{E:asymptotics}, almost surely for all large enough $n \in \N$ we have 
\begin{equation}
\label{E:A-W-relate}
A^{x_1, x_2}(z) = \max_{-n \le y \le z} W_\bx \cL_2(y) - W_\bx \cL_1(y)
\end{equation}
for all $z \in [a, b]$.
For each $n$, let $S^n_{x_1, x_2} \sset [-n, \infty)$ be the support of the measure corresponding to the CDF $\max_{-n \le y \le z} W_\bx \cL_2(y) - W_\bx \cL_1(y)$. Since $W_\bx \cL_2- W_\bx \cL_1$ is absolutely continuous with respect to Brownian motion on $[-n, b]$ by Theorem \ref{T:landscape-iso}.3, by Theorem 1 in \cite{taylor1966exact} (see also Section 6.4 in \cite{morters2010brownian}), there is an explicit function $H$ such that almost surely,
$$
H(S^n_{x_1, x_2} \cap [a, b]) = \max_{-n \le y \le b} [W_\bx \cL_2(y) - W_\bx \cL_1(y)] - \max_{-n \le y \le a} [W_\bx \cL_2(y) - W_\bx \cL_1(y)].
$$
(The exact nature of the function $H$ is not important; as an aside, it is type of generalized Hausdorff measure). By \eqref{E:A-W-relate}, for all large enough $n$ the right hand side above equals $A^{x_1, x_2}(b) - A^{x_1, x_2}(a)$ and $S^n_{x_1, x_2} \cap [a, b]= \mu_{x_1, x_2} \cap [a, b]$. Therefore a.s., 
$$
H(\mu_{x_1, x_2} \cap [a, b]) = A^{x_1, x_2}(b) - A^{x_1, x_2}(a).
$$
This holds simultaneously a.s. for all rationals $a <b, x_1 < x_2$. Next, a.s we have
\begin{align*}
A^{x_1, x_2}(a) &= \lim_{n \to \infty} \frac{1}{n} \sum_{i=1}^n A^{x_1, x_2}(a) - A^{x_1, x_2}(i) + \E A^{x_1, x_2}(i), \qquad \mathand \\
\cL(x_2, 0; a, 1) &= \lim_{n \to \infty} \frac{1}{n} \sum_{i=1}^n A^{x_2 -i, x_2}(a)  + \E \cL(x_2 - i, 0; a, 1)
\end{align*}
This follows from ergodicity of the stationary Airy processes $\cL(\cdot; 0, a, 1) + \E \cL(\cdot; 0, a, 1)$ and $\cL(x_i; 0, \cdot, 1) + \E \cL(x_i; 0, \cdot, 1)$, see equation (5.15) in \cite{prahofer2002scale}. Putting all this together gives that each of the points $\cL(x_2, 0; a, 1), x_2, a \in \Q$ is $\cF_{0, 1}$-measurable. Continuity of $\cL$ then gives the result.
\end{proof}

\begin{proof}[Proof of Fact \ref{F:KPZFP-structure} and Theorem \ref{T:KPZ-fixed-point}]
By scale invariance of $\cL$, we may assume $t = 1$. Let $\bp = (p_1, p_2)$. By Theorem \ref{T:landscape-iso}, we have 
\begin{equation}
\label{E:ht-L}
\begin{split}
\mathfrak{h}_t(y) &= \max \lf(\cL(p_1, 0; y, 1) + a_1, \cL(p_2, 0; y, 1) + a_2 \rg), \\
\cL(p_1, 0; y, 1) &=  W_\bp\cL_1(y), \\
\qquad \cL(p_2, 0; y, 1) &= W_\bp\cL_1(y) + \sup_{z \le y} [W_\bp\cL_2(z) - W_\bp\cL_1(z)].
\end{split}
\end{equation}
From this representation, it is clear that an $A \in [-\infty, \infty]$ satisfying the conditions of Fact \ref{F:KPZFP-structure} exists. The fact that $A \ne \pm \infty$ follows from the asymptotics in Theorem \ref{T:landscape-iso}.1. Next, fix $n \in \N$, and define $W^n_\bp\cL:[-n, \infty) \to \R^2$ by 
$$
W^n_\bp\cL(x) = \begin{cases}
W_\bp\cL(x) - W_\bp\cL(-n), \qquad x \in [-n, n] \\
W_\bp\cL(n) + B(x - n), \qquad x \ge n
\end{cases}
$$
where $B$ is an independent $2$-dimensional Brownian motion of variance $2$ with $B(0) = 0$. Consider the process $\mathfrak{h}_t^n = \max (M^n_1, M^n_2):[-n, \infty) \to \R$, where 
\begin{equation}
\label{E:Mn1}
\begin{split}
M^n_1(x) &= W^n_\bp\cL[(-n, 1) \to (x, 1)] + \max(a_1 + W_\bp \cL_1(-n), a_2 + W_\bp \cL_2(-n) + 1) , \\
M^n_2(x) &= W^n_\bp\cL[(-n, 2) \to (x, 1)] + a_2 + W_\bp \cL_2(-n).
\end{split}
\end{equation}
By Theorem \ref{T:landscape-iso}.2 and the spatial stationarity \eqref{E:translation-L} for $\cL$, conditional on $W_\bp \cL(-n)$ the process $W_\bp\cL:[-n, \infty) \to \R$ is absolutely continuous with respect to a $2$-dimensional Brownian motion of variance $2$ on $[-n, \infty)$ started from $B(-n) = 0$.
Adding the independent Brownian motion to the end of $W^n_\bp$ ensures that we can make the comparison on all of $[-n, \infty)$, rather than just on $[-n, n]$.

 Therefore conditional on $W_\bp \cL(-n)$, the processes $M^n_1, M^n_2$ satisfy the conditions of Proposition \ref{P:pre-kpzfp}. Here we have used that 
$$
\max(a_1 + W_\bp \cL_1(-n), a_2 + W_\bp \cL_2(-n) + 1) > a_2 + W_\bp \cL_2(-n).
$$
Therefore letting $\tau_n$ denote the maximal $x \in (-n, \infty)$ such that $\mathfrak{h}_t^n(x) = M^n_1(x)$, by Proposition \ref{P:pre-kpzfp} we have that
$$
(\tau_n, \mathfrak{h}_t^n|_{[-n, \infty)} - \mathfrak{h}_t^n(\tau_n)) \ll (X, [B + R](-X + \cdot)|_{[-n, \infty)}),
$$
where $B, R, X$ are as in the statement of Theorem \ref{T:KPZ-fixed-point}. Finally, comparing \eqref{E:ht-L} and \eqref{E:Mn1} and using the asymptotics in Theorem \ref{T:landscape-iso}.1, on any interval $I$, there exists a random $N \in \N$ such that
$$
(\tau_N, \mathfrak{h}_t^N|_{I} - \mathfrak{h}_t^N(\tau_N)) = (A, \mathfrak{h}_t|_{I} - \mathfrak{h}_t(A)),
$$
yielding the result.
\end{proof}

For Theorem \ref{T:main-comparison}, we first need an analogue of Proposition \ref{P:Wsig-extension} for side-to-side, rather than bottom-to-top, last passage values. We will need a notion of discrete last passage percolation. Consider an $m \X n$ array $G = G_{i, j}, i \in \II{1, m}, j \in \II{1, n}$. For $k \in \N$ and vectors $I, J \in \II{1, n}^k_\ge$ define the \textbf{last passage value}
\begin{equation*}
G[(1, I) \to (m, J)] = \sup_{\pi_1, \dots, \pi_k} \sum_{i=1}^k \sum_{v \in \pi} G_v.
\end{equation*}
Here the supremum is over all $k$-tuples of disjoint lattice paths $\pi_1, \dots, \pi_k$, where $\pi_i$ starts at $(1, I_i)$, ends at $(n, J_i)$ and only moves up and to the right in the coordinate system of Figure \ref{fig:dis-opt}, i.e. all of its steps are in $\{(0, -1), (1, 0)\}$. If no such disjoint paths exist, we set $G[(1, I) \to (n, J)] = -\infty$. For points $(1, i), (n, j)$ we write
$$
G[(1, i)^{*k} \to (m, j)^{*k}] 
$$
for the $k$-point last passage value from $(i, I)$ to $(n, J)$, where $I = (i + k -1, \dots, i +1, i), J = (j, j-1, \dots, j - k + 1)$. These definitions are analogous to the definitions of semi-discrete last passage introduced in Section 1. Moreover, discrete last passage percolation also satisfies an isometry theorem, see Section 8 of \cite{DNV2} and references therein to precursor theorems. For an $m \X n$ array $G$ with $m \ge n$, define an $n \X n$ array $WG$ by the system of equations
\begin{equation*}
G[(1, n)^{*(n + 1 - k \vee \ell)} \to (m, \ell)^{*(n + 1 - k \vee \ell)}] = \sum_{i=k}^n \sum_{j=\ell}^n WG_{i, j}
\end{equation*}
for $\ell, k \in \II{1, n}$. Then we have the side-to-side isometry
\begin{equation}
\label{E:GWG}
G[(1, I) \to (m, J)] = WG[(1, I) \to (n, J)]
\end{equation}
for any $I, J$.

Now suppose $f$ is a sequence of $n$ continuous functions from $[0, t] \to \R$. For all $m \in \N$, let $f^m$ be the $m \X n$ array where $f^m_{i, j} = f^m_j(ti/m) - f^m_j(t(i-1)/m)$. Then since $f$ is continuous it is easy to check that as $m \to \infty$, for any $I, J$ we have
\begin{align*}
f^m[(1, I) \to (m, J)] - f[(0, I) \to (t, J)] \to 0
\end{align*}
and for any $i, j$ we have
\begin{align*}
f^m[(1, i)^{*k} \to (m, j)^{*k}] - f[(0, i)^{k} \to (t, j)^{k}] \to 0.
\end{align*}
Combining these facts with \eqref{E:GWG} immediately gives a side-to-side isometry for continuous functions.
\begin{prop}
	\label{P:F-side-to-side}
	Let $f = (f_i: \R \to \R, i \in \II{1, n})$ be a sequence of continuous functions, let $t > 0$, and define an $n \X n$ array $W^t f$ given by
	$$
	f[(0, n)^{n + 1 - k \vee \ell} \to (t, \ell)^{n + 1 - k \vee \ell}] = \sum_{i=k}^n \sum_{j=\ell}^n W^tf_{i, j}.
	$$
	Then for any $I, J \in \II{1, n}^k_\ge$ we have
	$$
	f[(0, I) \to (t, J)] = W^t f[(1, I) \to (n, J)].
	$$
\end{prop}

We use this proposition to prove the following lemma.

\begin{lemma}
	\label{L:ac-lemma}
	For a sequence of $n$ continuous functions $f$ and $t > 0$ define $g_{t, f}:\{(i, j) \in \II{1, n} : i \ge j\} \to \R$ by
	$$
	g_{t,f}(i, j) = f[(0, i) \to (t, j)].
	$$
	Then for any $0 < s,t$, if $B$ is a sequence of independent standard Brownian motions, then $g_{t, B}$ and $g_{s, B}$ are mutually absolutely continuous.
\end{lemma}

\begin{proof}
	By Proposition \ref{P:F-side-to-side}, it is enough to show that $W^t B$ and $W^s B$ are mutually absolutely continuous for all $t, s$. The array $W^t B$ contains exactly the same information as the triangular array $X^t = \{X^t_{i, j}, i \le j \in \II{1, n}\}$ given by the rule
	\begin{equation*}
	\sum_{i=1}^k X^t_{i, j} = f[(1, n)^k \to (t, n - j + 1)^k].
	\end{equation*}
	To complete the proof, it enough to note that for any $t$, the law of $X^t$ is mutually absolutely continuous with respect to Lebesgue measure on the set of all arrays $x=x_{i, j}$ satisfying the interlacing inequalities
	\begin{equation}
	\label{E:Xij}
	x_{i, j} \ge x_{i, j-1}, \qquad x_{i, j-1} \ge x_{i+1, j}
	\end{equation}
	for all $i, j$ (i.e. Lebesgue measure on Gelfand-Tsetlin patterns). This is well-known. Indeed, $X^t/\sqrt{t}$ has the same law as the GUE minors eigenvalue process $\la$, see p. 3691 in \cite{o2003path}. 
	The top row of GUE eigenvalues in this process has an explicit Lebesgue density (e.g. see \cite[Theorem 2.5.2]{anderson2010introduction}) and can hence be seen to be mutually absolutely continuous with respect to Lebesgue measure on the set where $\la_{1, n} \ge \dots \ge \la_{n, n}$. Conditional on this top row, the law of the remaining rows is uniform on the compact set satisfying the inequalities in \eqref{E:Xij}, see \cite{baryshnikov2001gues}.
\end{proof}

\begin{proof}[Proof of Theorem \ref{T:main-comparison}]
	By combining parts $1$ and $3$ of Theorem \ref{T:landscape-iso}, we get that there is a random integer $N \le - b-1$ such that for all $i \in \{1, \dots, k\}$ and $y \in [-b, b]$ we have
	$$
	\cL(x_i, 0; y, 1) - \cL(x_i, 0; -b, 1) = W_\bx \cL[(N, i) \to (-b, 1)] - W_\bx \cL[(N, i) \to (y, 1)].
	$$
	By Theorem \ref{T:landscape-iso}.2, for every integer $n \le -b -1$, the process 
	$$
	W_\bx \cL[(n, i) \to (-b, 1)] - W_\bx \cL[(n, i) \to (y, 1)], i \in \{1, \dots, k\}, y \in [-b, b]
	$$ is absolutely continuous with with respect to the same last passage process with $W_\bx \cL$ replaced by $n$ independent Brownian motions $B$:
	\begin{equation}
	\label{E:BniBni}
	B[(n, i) \to (-b, 1)] - B[(n, i) \to (y, 1)], i \in \{1, \dots, k\}, y \in [-b, b].
	\end{equation}
	Therefore to complete the proof, we just need to check that for all integers $n \ge - b-1$, the process in \eqref{E:BniBni} is absolutely continuous with respect to the same process with $n$ replaced by $-b-1$. We will check the stronger claim that the processes
	\begin{equation}
	\label{E:Bai}
	B[(a, i) \to (y, 1)], i \in \{1, \dots, k\}, y \in [-b, b]
	\end{equation}
	are all mutually absolutely continuous for $a < b$. Indeed, the laws of the random vectors 
	\begin{equation}
	\label{E:vector}
	(B[(a, i) \to (b, j)], j \le i \in \II{1, k})
	\end{equation} are mutually absolutely continuous for all $a < b$ by Lemma \ref{L:ac-lemma}. Therefore to complete the proof, it is enough to show that conditional on the vector in \eqref{E:vector}, the distribution in \eqref{E:Bai} does not depend on $a$. Under this conditioning the function $B|_{[-b, b]}$ is still just a collection of $n$ independent Brownian motions, and by the metric composition law (Proposition \ref{P:MC-law}) at the vertical line $\{b\}\X \Z$, for $i \in \II{1, k}, y \in [-b, b]$ we have
	\begin{equation*}
	B[(a, i) \to (y, 1)] = \max_{j \in \II{1, i}} B[(a, i) \to (b, j)] + B[(b, j) \to (y, 1)]. 
	\end{equation*}
	The right hand side is a function of \eqref{E:vector} and $B|_{[-b, b]}$, completing the proof.
\end{proof}

\bibliographystyle{dcu}

\bibliography{bibliography}
\end{document}